\documentclass[12pt,CJK]{article}
\usepackage{amsfonts,amsmath,amssymb,amsthm,mathrsfs}
\usepackage{latexsym}
\usepackage{graphicx}
\usepackage{indentfirst}
\usepackage{color}
\usepackage{fancyhdr}
\usepackage[colorlinks,linktocpage,linkcolor=blue]{hyperref}

\theoremstyle{plain}
\newtheorem{thm}{Theorem}[section]
\newtheorem{definition}{Definition}
\newtheorem{lemma}[thm]{Lemma}
\newtheorem{corollary}[thm]{Corollary}
\newtheorem{remark}[thm]{Remark}
\numberwithin{equation}{section}

\theoremstyle{remark}

\def\Xint#1{\mathchoice
  {\XXint\displaystyle\textstyle{#1}}%
  {\XXint\textstyle\scriptstyle{#1}}%
  {\XXint\scriptstyle\scriptscriptstyle{#1}}%
  {\XXint\scriptscriptstyle\scriptscriptstyle{#1}}%
  \!\int}
\def\XXint#1#2#3{{\setbox0=\hbox{$#1{#2#3}{\int}$}
  \vcenter{\hbox{$#2#3$}}\kern-.5\wd0}}

\def\dashint{\Xint-}

\begin{document}
\allowdisplaybreaks
\pagestyle{myheadings}\markboth{$~$ \hfill {\rm Q. Xu,} \hfill $~$} {$~$ \hfill {\rm  } \hfill$~$}
\author{Li Wang
\quad Qiang Xu
\thanks{Corresponding author}
\thanks{Email: xuqiang09@lzu.edu.cn.}
\quad Peihao Zhao
\\
School of Mathematics and Statistics, Lanzhou University, \\
Gansu, 730000, PR China. \vspace{0.5cm}
}

%


\title{\textbf{Quantitative Estimates on Periodic Homogenization
of Nonlinear Elliptic Operators} }
\maketitle
\begin{abstract}
In this paper, we are interested in the periodic homogenization of quasilinear elliptic equations.
We obtain error estimates $O(\varepsilon^{1/2})$ for a $C^{1,1}$ domain,
and $O(\varepsilon^\sigma)$ for a Lipschitz domain,
in which $\sigma\in(0,1/2)$ is
close to zero. Based upon the convergence rates,
an interior Lipschitz estimate, as well as
a boundary H\"older estimate can be developed at large scales without any smoothness
assumption, and these
will implies reverse H\"older estimates established for a $C^1$ domain.
By a real method developed by Z.Shen \cite{S3},
we consequently derive a global $W^{1,p}$ estimate for $2\leq p<\infty$.
This work may be regarded as an extension of \cite{MAFHL,S5} to
a nonlinear operator, and our results may be extended to
the related Neumann boundary problems without any real difficulty.
\end{abstract}

\section{Instruction and main results}
The aim of the present paper is to study the error estimates
and $W^{1,p}(\Omega)$ estimates with $2\leq p\leq \infty$ for quasilinear elliptic
equations, arising in the periodic homogenization theory.
More precisely, let $\Omega\subset\mathbb{R}^d$ be a bounded
domain, and consider the following elliptic equations in divergence form
depending on a parameter $\varepsilon > 0$,
\begin{eqnarray}\label{pde:1.1}
\left\{\begin{aligned}
\mathcal{L}_{\varepsilon} u_\varepsilon \equiv
-\text{div}A(x/\varepsilon,\nabla u_\varepsilon) &= F &\qquad&\text{in}~~\Omega, \\
 u_\varepsilon &= g &\qquad& \text{on}~\partial\Omega.
\end{aligned}\right.
\end{eqnarray}

Given three constants $\mu_0,\mu_1,\mu_2>0$, let us fix a function
$A:\mathbb{R}^d\times\mathbb{R}^d\to\mathbb{R}^d$ which satisfies the following conditions.
\begin{itemize}
\item For every $z\in\mathbb{R}^d$, $A(\cdot,z)$ is 1-periodic and Lebesgue measurable, and
\begin{equation}\label{a:1}
  A(y,0) = 0  \quad\text{for~a.e.}~y\in\mathbb{R}^d.
\end{equation}
\item There hold the coerciveness and growth conditions
\begin{equation}\label{a:2}
\begin{aligned}
&\big<A(y,z)-A(y,z^\prime),z-z^\prime\big>  \geq \mu_0|z-z^\prime|^2, \\
&~|A(y,z)-A(y,z^\prime)| \leq \mu_2|z-z^\prime|,
\end{aligned}
\end{equation}
and the smoothness condition
\begin{equation}\label{a:3}
|A(y,z)-A(y^\prime,z)|\leq \mu_1|y-y^\prime|^\tau
\end{equation}
for all $y,y^\prime\in\mathbb{R}^d$ and $z,z^\prime\in\mathbb{R}^d$, where $\tau\in(0,1]$.
\end{itemize}

The following qualitative homogenization results are well known (see for example \cite{A,MD}).
Let $F\in H^{-1}(\Omega)$ and let
$u_\varepsilon$ be the weak solution to problem $\eqref{pde:1.1}$.
Then we have $u_\varepsilon \rightharpoonup u_0$ weakly in $H^1(\Omega)$, and
$A(x/\varepsilon,\nabla u_\varepsilon)\rightharpoonup\widehat{A}(\nabla u_0)$ weakly in $L^2(\Omega;\mathbb{R}^d)$,
where $u_0$ is the solution to the effective (homogenized) equation
\begin{equation}\label{pde:1.3}
\left\{\begin{aligned}
\mathcal{L}_{0} u_0 \equiv
-\text{div}\widehat{A}(\nabla u_0) &= F &\qquad&\text{in}~~\Omega, \\
 u_0 &= g &\qquad& \text{on}~\partial\Omega.
\end{aligned}\right.
\end{equation}
The function $\widehat{A}:\mathbb{R}^d\to\mathbb{R}^d$ is defined for every $\xi\in\mathbb{R}^d$ by
\begin{equation}\label{eq:1.1}
\widehat{A}(\xi) = \int_{Y=(0,1]^d} A(y,\xi+\nabla_y N(y,\xi))dy,
\end{equation}
where $N(y,\xi)$ is the so-called corrector, satisfying the following cell problem
\begin{equation}\label{pde:1.2}
\left\{\begin{aligned}
&\text{div} A^\xi(y,\nabla_y N(y,\xi)) = 0 \quad\text{in}~ Y,\\
& N(\cdot,\xi)\in H^1_{per}(Y), \quad \dashint_{Y}N(\cdot,\xi) = 0,
\end{aligned}\right.
\end{equation}
where $A^\xi(y,z) = A(y,\xi+z)$ for any $z\in\mathbb{R}^d$.
The notation $\dashint_\Omega = \frac{1}{|\Omega|}\int_\Omega$,
where $|\Omega|$ represents
the volume of $\Omega$. Let $r_0$ denote the diameter of $\Omega$ throughout the paper.

We now state the main results of the paper.
\begin{thm}[convergence rates]\label{thm:1.1}
Suppose that $\mathcal{L}_\varepsilon$ satisfies the conditions
$\eqref{a:1}$ and $\eqref{a:2}$. Let
$u_\varepsilon, u_0\in H^1(\Omega)$ be the weak solution of
$\eqref{pde:1.1}$ and $\eqref{pde:1.3}$, respectively.
\begin{itemize}
  \item[\emph{(1).}] Let $\Omega\subset\mathbb{R}^d$
  be a bounded $C^{1,1}$ domain with $r_0=\emph{diam}(\Omega)$,
  and $F\in L^2(\Omega)$ and $g\in H^{3/2}(\partial\Omega)$, then we have
\begin{equation}\label{pri:1.2}
\|u_\varepsilon - u_0\|_{L^2(\Omega)}
\leq C\varepsilon^{1/2}\Big\{\|F\|_{L^2(\Omega)}
+\|g\|_{H^{3/2}(\partial\Omega)}
\Big\},
\end{equation}
in which $C$ depends on $\mu_0,\mu_2,d,r_0$ and the character of $\Omega$.
  \item[\emph{(2).}] If $\Omega$ is a bounded Lipschitz domain, then there exists
  $p>2$, such that
\begin{equation}\label{pri:1.3}
\|u_\varepsilon - u_0\|_{L^2(\Omega)}
\leq C_p\varepsilon^{\sigma}\Big\{\|F\|_{L^p(\Omega)}
+\|g\|_{W^{1-1/p,p}(\partial\Omega)}
\Big\},
\end{equation}
where $\sigma = 1/2-1/p$, and $C_p$
depends on $\mu_0,\mu_2,d,p,r_0$ and the character of $\Omega$.
\end{itemize}
\end{thm}

\begin{thm}[interior Lipschitz estimates at large scales]\label{thm:1.2}
Let $B(0,2)\subset\Omega$.
Suppose that the coefficients of $\mathcal{L}_\varepsilon$ satisfy
$\eqref{a:1}$ and $\eqref{a:2}$.
Let $u_\varepsilon\in H^1(B(0,2);\mathbb{R}^m)$ be a weak solution of
$\mathcal{L}_\varepsilon u_\varepsilon = F$ in $B(0,2)$,
where $F\in L^p(B(0,2))$ with $p>d$.
Then there holds
\begin{equation}\label{pri:1.1}
\begin{aligned}
\Big(\dashint_{B(0,r)} |\nabla u_\varepsilon|^2dx\Big)^{\frac{1}{2}}
\leq C\bigg\{\Big(\dashint_{B(0,1)} |\nabla u_\varepsilon|^2dx\Big)^{\frac{1}{2}}
+ \Big(\dashint_{B(0,1)} |F|^pdx\Big)^{\frac{1}{p}}\bigg\}
\end{aligned}
\end{equation}
for any $\sqrt[3]{\varepsilon}\leq r<(1/4)$,
where $C$ depends only on $\mu_0, \mu_2,p$ and $d$.
\end{thm}

\begin{thm}[$W^{1,p}$ estimates]\label{thm:1.3}
Let $\Omega$ be a bounded $C^1$ domain and $2\leq p<\infty$. Suppose that
$\mathcal{L}_\varepsilon$ satisfies $\eqref{a:1}$, $\eqref{a:2}$ and
$\eqref{a:3}$. Let $F\in W^{-1,p^\prime}(\Omega)$ and
$g\in W^{1-1/p,p}(\partial\Omega)$
with $1/p+1/p^\prime = 1$. Then
the weak solution $u_\varepsilon$ of $\eqref{pde:1.1}$ admits the following estimate
\begin{equation}\label{pri:1.4}
\|\nabla u_\varepsilon\|_{L^p(\Omega)}
\leq C\Big\{\|F\|_{W^{-1,p^\prime}(\Omega)}
+\|g\|_{W^{1-1/p,p}(\partial\Omega)}\Big\},
\end{equation}
where $C$ depends on $\mu_0,\mu_1,\mu_2,\tau,d,p,r_0$ and the character of $\Omega$.
\end{thm}

A few remarks are in order.

\begin{remark}
\emph{To see the existence of the solution of
$\eqref{pde:1.2}$, let $\mathcal{L}_1^\xi v = -\text{div}A^\xi(y,\nabla v)$ in $Y$,
and we first observe that $A^\xi(y,z)$ satisfies
the condition $\eqref{a:1}$ and $A^{\xi}(y,-\xi) = 0$ for any $z,\xi\in\mathbb{R}^d$. This together with
the assumption $\eqref{a:2}$ verifies
that $\mathcal{L}_1^\xi: \mathcal{W}_{per}(Y) \to (\mathcal{W}_{per}(Y))^*$ is
a strong monotone, hemicontinuous and coercive operator, in which
$\mathcal{W}_{per}(Y) = \{u\in H_{per}^1(Y): \int_{Y}u(y) dy = 0\}$ and
$(\mathcal{W}_{per}(Y))^*$ denotes its dual space.
Thus it is well known (see for example \cite[Theorem 26.A]{Z}) that $\mathcal{L}_1^\xi v = 0$
has a unique solution $v\in \mathcal{W}_{per}(Y)$ for any $\xi\in\mathbb{R}^d$, denoted by
$N(y,\xi)$ in $\eqref{pde:1.2}$. The definition of the space
$H_{per}^1(Y)$ may be found in \cite{ABJLGP,VSO}.}
\end{remark}

\begin{remark}
\emph{Although the estimate $\eqref{pri:1.2}$ has already been shown by S. Pastukhova\cite{Z},
our arguments do not rely on the related error estimate in the whole space as
a precondition compared to his, and are also valid for deriving the estimate
$\eqref{pri:1.3}$ in terms of a bounded Lipschitz domain.
The key ingredient is to reduce
the corresponding problem to the so-called ``layer'' and ``co-layer''
type estimates
\begin{equation*}
\|\nabla u_0\|_{L^2(\Omega\setminus \Sigma_{4\varepsilon})}
\quad\text{and}\quad
\|\nabla^2 u_0\|_{L^2(\Sigma_{4\varepsilon})},
\end{equation*}
where $\Sigma_{4\varepsilon} =
\big\{x\in\Omega:\text{dist}(x,\partial\Omega>4\varepsilon)\big\}$, and then
we employ $H^2$,
and Meyers estimates (see Theorems $\ref{thm:2.2}$,$\ref{thm:2.3}$) to obtain the stated estimates
$\eqref{pri:1.2}$ and $\eqref{pri:1.3}$, respectively. To begin the proof, we
define the first order corrector as
\begin{equation*}
v_\varepsilon(x) = u_0(x) + \varepsilon N(x/\varepsilon,S_\varepsilon(\psi_{4\varepsilon}\nabla u_0)),
\end{equation*}
where $S_\varepsilon$ is a smoothing operator (see Definition $\ref{def:2.1}$), and
$\psi_{4\varepsilon}$ is a cut-off function (see $\eqref{def:3.1}$).
We mention that the idea is inspired by the so-called shift argument
introduced by V. Zhikov, S. Pastukhova in \cite{ZVVPSE}.
Here, we substitute smoothing operator
$S_\varepsilon$ for the Steklov averaging, which was first suggested by
Z. Shen \cite{S5}. In terms of linear operators,
error estimates have been studied extensively, and
we refer the reader to \cite{GX,KFS2,S5,TS1,X0,X3,ZVVPSE,ZVVPSE1} and their references therein
for more results.}
\end{remark}

\begin{remark}
\emph{If the monotone operator $\mathcal{L}_\varepsilon$ additionally satisfies
the condition $\eqref{a:3}$, then the estimate $\eqref{pri:1.1}$ will imply
the uniform Lipschitz estimate
\begin{equation}\label{pri:1.6}
 \|\nabla u_\varepsilon\|_{L^\infty(B(0,r/2))}
 \leq C\bigg\{\Big(\dashint_{B(0,1)} |\nabla u_\varepsilon|^2dx\Big)^{\frac{1}{2}}
+ \Big(\dashint_{B(0,1)} |F|^pdx\Big)^{\frac{1}{p}}\bigg\}
\end{equation}
for any $0<r<(1/4)$. This type result was first obtained by M. Avellaneda, F. Lin
\cite{MAFHL} for the linear case $A(y,\xi)=A(y)\xi$,
in which a compactness method had been well developed. However,
their method can not be applied to non-periodic setting
or nonlinear operators, directly.
Recently, S. Armstrong and Z. Shen derived the estimate $\eqref{pri:1.1}$ for
almost-periodic homogenization of linear elliptic systems with either Dirichlet
or Neumann boundary conditions, while
S. Armstrong, J. Mourrat, C. Smart \cite{AM,ASC} obtained the estimate
$\eqref{pri:1.1}$ for stochastic homogenization of convex integral functionals.
Their new idea is based upon a convergence rate coupled with the so-called
Campanato iteration. In this sense, either the estimate $\eqref{pri:1.2}$ or
$\eqref{pri:1.3}$ will lead to the stated estimate $\eqref{pri:1.1}$.
Inspired by some techniques in \cite{SZ1}, we plan to use the estimate
$\eqref{pri:1.2}$ here. Since
we can not use the rescaling argument, we have to calculate carefully, and first
obtain a approximating lemma
\begin{equation*}
 \Big(\dashint_{B(0,r)} |u_\varepsilon - w|^2  \Big)^{1/2}
 \leq C\left(\frac{\varepsilon}{r}\right)^{1/4}
 \bigg\{
 \Big(\dashint_{B(0,2r)}|u_\varepsilon|^2 \Big)^{1/2}
 + r^2\Big(\dashint_{B(0,2r)} |F|^2\Big)^{1/2}\bigg\}
\end{equation*}
for $\sqrt[3]{\varepsilon}\leq r<(1/4)$, where $w\in H^1(B(0,r_0))$ satisfies
$\mathcal{L}_0 w = F$ in $B(0,2r)$ with $w=u_\varepsilon$ on $\partial B(0,r_0)$,
and $r<r_0<2r$. Then we use the iteration argument (see Lemma $\ref{lemma:4.3}$) to prove our result,
which was proved by Z. Shen in \cite{S5}, originally shown in \cite{AM,ASC}. To carry out
this program, we define the following quantity
\begin{equation*}
G(r,u_\varepsilon) = \frac{1}{r}\inf_{M\in\mathbb{R}^{d}\atop c\in\mathbb{R}}
\Bigg\{\Big(\dashint_{B(0,r)}|u_\varepsilon-Mx-c|^2dx\Big)^{\frac{1}{2}}
+ r^2\Big(\dashint_{B(0,r)}|F|^p\Big)^{\frac{1}{p}}\Bigg\},
\end{equation*}
in which it is not hard to see that the requirement $p>d$ is natural assumption.
Although the quantity $w-Mx-c$ is not a solution of $\mathcal{L}_0 w = F$ in
$B(0,r_0)$ in general, the key observation is that it verified
the same linearized equation as $w$ did. Thus we can show there exists
$\theta\in(0,1/4)$, depending on $\mu_0,\mu_2,d$, such that
\begin{equation*}
 G(\theta r, w)\leq (1/2)G(r,w)
\end{equation*}
for any $0<r<1$ (see Lemma $\ref{lemma:4.2}$). Then the rest of the proof is
standard. We point out that due to
the above result depending on De Giorgi-Nash-Moser theorem, we can not
extend the estimate $\eqref{pri:1.1}$ to the corresponding systems.
We also mention that
if we use the error estimate $\eqref{pri:2.3}$, the estimate $\eqref{pri:1.1}$
may hold for $\varepsilon\leq r<(1/4)$.
The sharp quantitative estimates received a great amount of interest recently,
and without attempting to
exhaustive, we refer the reader
to \cite{GNF,GX,KFS1,MN,NSX,S1,X0,X4} and references therein for more results.}
\end{remark}

\begin{remark}
\emph{The estimate $\eqref{pri:1.4}$ may be divided into
the corresponding interior and boundary estimates. The first step is to
derive the reverse H\"older estimates for the homogeneous case (see the estimates
\eqref{f:4.10} and $\eqref{pri:5.6}$),
which are based upon classical results
\cite[Theorem 1.4]{BW} in small scales and the Lipschitz estimate $\eqref{pri:1.1}$,
the boundary H\"older estimate $\eqref{pri:5.4}$, at large scales, respectively. Then one may
extend reverse H\"older estimates to the nonhomogeneous cases
(also including nonhomogeneous boundary condition), which requires
a real method developed by Z. Shen \cite{S3}, inspired by \cite{CP}.
We need to mention that the estimate $\eqref{f:4.10}$ for a linear operator
has been shown by L. Caffarelli and I. Peral  in \cite{CP}.
Although the real method has been applied to studying $W^{1,p}$ estimates or
$L^p$ boundary value problems by a lot of papers
(see for example \cite{GJ,KFS1,NSX,S1,S5,S3,X4}),
to our best acknowledge, it is the first time to be used to
a nonlinear operator.
Due to Theorems $\ref{thm:1.1}$ and $\ref{thm:1.2}$, we also mention that it is possible to derive the related results for
Reifenberg flat domains, via the arguments introduced by
S. Byun and L. Wang \cite{SBLW,BW}.}
\end{remark}

\begin{corollary}
Assume the same conditions as in Theorem $\ref{thm:1.3}$, and
$f\in L^p(\Omega;\mathbb{R}^d)$ with $p>d$.
Then we have the uniform H\"older estimate
\begin{equation}\label{pri:1.5}
\|u_\varepsilon\|_{C^{0,\alpha}(\Omega)}
\leq C\Big\{\|f\|_{L^p(\Omega)}+\|g\|_{C^{0,1}(\partial\Omega)}\Big\},
\end{equation}
where $\alpha = 1-d/p$, and $C$
depends on $\mu_0,\mu_1,\mu_2,d,p,r_0$ and the character of $\Omega$.
\end{corollary}

The paper is organized in five sections.
In Section 2, we establish some properties on correctors,
and the required preliminary knowledge. In Section 3, we investigate
convergence rates for different type domains. We present the proof of
Theorem $\ref{thm:1.2}$, $\ref{thm:1.3}$ in Sections 4,5, respectively.

\section{Preliminaries}

\begin{lemma}\label{lemma:2.1}
Let $N(\cdot,\xi)\in H_{per}^1(Y)$ be the weak solution to the equation $\eqref{pde:1.2}$, and then
for any $\xi\in\mathbb{R}^d$, we have the following estimates
\begin{equation}\label{pri:2.1}
 \dashint_{Y} |N(\cdot,\xi)|^2 + \dashint_{Y} |\nabla N(\cdot,\xi)|^2 \leq C|\xi|^2
\end{equation}
and
\begin{equation}\label{pri:2.2}
\dashint_{Y} |\nabla_\xi N(\cdot,\xi)|^2 + \dashint_{Y} |\nabla_{\xi}\nabla N(\cdot,\xi)|^2 \leq C,
\end{equation}
where $C$ depends only on $\mu_0,\mu_2$ and $d$.
\end{lemma}

\begin{proof}
These results have already been in \cite{P}, and we provide a proof for the sake of the completeness.
Multiplying both sides of $\eqref{pde:1.2}$ by $N(y,\xi)$ and then integrating by parts, we have
\begin{equation*}
\begin{aligned}
0 &= \int_{Y} A^{\xi}(y,\nabla_y N(y,\xi))\cdot\nabla_yN(y,\xi) dy \\
 &= \int_{Y} A(y,\xi+\nabla_y N(y,\xi))\cdot\big(\xi+\nabla_yN(y,\xi)\big)dy
 - \int_{Y} A(y,\xi+\nabla_y N(y,\xi))dy\cdot\xi \\
 &\geq \mu_0 \int_Y |\xi + \nabla_y N(y,\xi)|^2 dy - \mu_2|\xi|\int_Y|\xi+\nabla_y N(y,\xi)|dy,
\end{aligned}
\end{equation*}
where we use the assumptions $\eqref{a:2}$ and $\eqref{a:3}$ in the last inequality.
By Young's inequality,
\begin{equation*}
\int_Y |\xi + \nabla_y N(y,\xi)|^2 dy \leq C(\mu_0,\mu_2)|\xi|^2.
\end{equation*}
Thus this together with Poincar\'e's inequality will give the stated estimate $\eqref{pri:2.1}$.

To show the estimate $\eqref{pri:2.2}$, we start with the following identity
\begin{equation}\label{f:2.1}
\begin{aligned}
&\int_Y \big[A(y,\xi+\nabla_y N(y,\xi)) - A(y,\xi^\prime+\nabla_y N(y,\xi^\prime))\big]\cdot
\big[\xi-\xi^\prime + \nabla_yN(y,\xi) - \nabla_y N(y,\xi^\prime)\big]dy \\
& = \int_{Y} \big[A(y,\xi+\nabla_y N(y,\xi)) - A(y,\xi^\prime+\nabla_y N(y,\xi^\prime))\big] dy
\cdot \big(\xi-\xi^\prime\big)
\end{aligned}
\end{equation}
where we use the fact that $N(\cdot,\xi),N(\cdot,\xi^\prime)\in H^1_{per}(Y)$ satisfy the equation
$\eqref{pde:1.2}$ for $\xi,\xi^\prime\in\mathbb{R}^d$, respectively. By the assumption $\eqref{a:2}$,
the left-hand side above is greater than
\begin{equation*}
\mu_0\int_Y |\xi-\xi^\prime + \nabla_yN(y,\xi) - \nabla_y N(y,\xi^\prime)|^2 dy,
\end{equation*}
while it follows from $\eqref{a:3}$ and Young's inequality that its right-hand side is less than
\begin{equation*}
\frac{\mu_0}{2}\int_Y |\xi-\xi^\prime + \nabla_yN(y,\xi) - \nabla_y N(y,\xi^\prime)|^2 dy
+ C(\mu_0,\mu_2)|\xi-\xi^\prime|^2.
\end{equation*}
Thus it is not hard to derive that
\begin{equation}\label{f:2.2}
\Big(\int_Y |\nabla_yN(y,\xi) - \nabla_y N(y,\xi^\prime)|^2 dy\Big)^{1/2} \leq C|\xi-\xi^\prime|,
\end{equation}
and this will give the estimate $\eqref{pri:2.2}$ in a similar way. We have completed the proof.
\end{proof}

\begin{remark}\label{remark:2.1}
\emph{In view of the estimate $\eqref{pri:2.1}$, one may conclude that $N(y,0) = 0$ for a.e. $y\in\mathbb{R}^d$.}
\end{remark}

\begin{lemma}
Suppose $\mathcal{L}_\varepsilon$ satisfies the assumptions $\eqref{a:1}$, $\eqref{a:2}$ and $\eqref{a:3}$.
Let $\widehat{A}$ be given in $\eqref{eq:1.1}$.
Then the effective operator $\mathcal{L}_0$
is still strongly monotone, coercive, satisfying the same growth condition, i.e,
\begin{equation}\label{pri:2.3}
\left\{\begin{aligned}
&\big<\widehat{A}(\xi)-\widehat{A}(\xi^\prime),\xi-\xi^\prime\big>\geq \mu_0|\xi-\xi^\prime|^2,\\
& |\widehat{A}(\xi)-\widehat{A}(\xi^\prime)|\leq C|\xi-\xi^\prime|,\\
& \widehat{A}(0) = 0,
\end{aligned}\right.
\end{equation}
where $C$ depends on $\mu_0,\mu_2$ and $d$.
\end{lemma}

\begin{proof}
The proof may be found in \cite{P}, and we provide a proof for the sake of completeness.
Due to the formula $\eqref{f:2.1}$, we have
\begin{equation*}
\begin{aligned}
\big<\widehat{A}(\xi)-\widehat{A}(\xi^\prime),\xi-\xi^\prime\big>
&\geq \mu_0\int_Y |\xi-\xi^\prime + \nabla_y(N(y,\xi) - N(y,\xi^\prime))|^2 dy,\\
&\geq \mu_0\int_Y |\xi-\xi^\prime|^2 dy,
\end{aligned}
\end{equation*}
where we use the fact that $\int_{\partial Y}\big[N(y,\xi) - N(y,\xi^\prime)\big]dS = 0$ for
$N(\cdot,\xi),N(\cdot,\xi^\prime)\in H_{per}^1(Y)$ have the same periodicity.
Note that
\begin{equation*}
\begin{aligned}
|\widehat{A}(\xi)-\widehat{A}(\xi^\prime)|&\leq
\int_Y\big|A(y,\xi+\nabla N(y,\xi))-A(y,\xi^\prime+\nabla N(y,\xi^\prime))\big| dy \\
&\leq \mu_2 \int_Y |\xi-\xi^\prime + \nabla N(y,\xi) - \nabla N(y,\xi^\prime)|dy \\
&\leq C|\xi-\xi^\prime|,
\end{aligned}
\end{equation*}
in which the last step is due to the estimate $\eqref{pri:2.2}$.
In view of Remark $\ref{remark:2.1}$, we may have the third line of
$\eqref{pri:2.3}$ and the proof is complete.
\end{proof}

\begin{remark}
\emph{Due to the second line of $\eqref{pri:2.3}$,
it is known that $\nabla \widehat{A}(z)$ exists for a.e. $z\in\mathbb{R}^d$. Moreover,
there holds
\begin{equation}\label{a:4}
\sum_{i,j=1}^d\nabla_j\widehat{A}_i(z)\xi_j\xi_i
 = \lim_{t\to 0} \frac{\big<\widehat{A}(z+t\xi)-\widehat{A}(z),\xi\big>}{t}
 \geq \mu_0|\xi|^2
\end{equation}
for any $\xi\in\mathbb{R}^d$ and for a.e. $z\in\mathbb{R}^d$,
and this property will guarantee that
the $H^2$ theory is still valid for the effective operator $\mathcal{L}_0$.}
\end{remark}

\begin{lemma}[Flux correctors]
Suppose $A$ satisfies $\eqref{a:1}$, $\eqref{a:2}$ and $\eqref{a:3}$.
Let $b(y,\xi) = A(y,\xi+\nabla N(y,\xi)) - \widehat{A}(\xi)$, where $y\in Y$ and $\xi\mathbb{R}^d$.
Then we have two properties: \emph{(i)} $\dashint_Y b(\cdot,\xi) = 0$;
\emph{(ii)} $\emph{div} b(\cdot,\xi) = 0$ in $Y$. Moreover, there
exists the so-called flux corrector $E_{ji}(\cdot,\xi)\in H^1_{per}(Y)$ such that
\begin{equation}\label{eq:2.1}
 b_i(y,\xi) = \frac{\partial}{\partial y_j}\big\{E_{ji}(y,\xi)\big\}
 \qquad \text{and} \quad E_{ji} = -E_{ij},
\end{equation}
and
\begin{equation}\label{pri:2.4}
\dashint_Y |\nabla_\xi E_{ji}(\cdot,\xi)|^2 + \dashint_Y |\nabla_\xi \nabla E_{ji}(\cdot,\xi)|^2 \leq C,
\end{equation}
where $C$ depends only on $\mu_0,\mu_2$ and $d$.
\end{lemma}

\begin{proof}
The proof is quite similar to the linear case (see for example \cite{S4,ZVVPSE1}).
It is clear to see that (i) and (ii) follow from the formula $\eqref{eq:1.1}$
and the equation $\eqref{pde:1.2}$, respectively. By (i),
there exists $f_i(\cdot,\xi)\in H_{per}^{2}(Y)$ such that $\Delta f_i(\cdot,\xi) = b_i(\cdot,\xi)$ in $Y$.
Let $E_{ji}(y,\xi)= \frac{\partial}{\partial y_j}\big\{f_{i}(y,\xi)\big\}
-\frac{\partial}{\partial y_i}\big\{f_{j}(y,\xi)\big\}$. Thus $E_{ji} = - E_{ij}$, and
one may derive the first expression in $\eqref{eq:2.1}$ from the fact $(\text{ii})$.
Then, the rest thing is to show the estimate $\eqref{pri:2.4}$.
For any $\xi,\xi^\prime\in\mathbb{R}^d$, note that
\begin{equation*}
\begin{aligned}
\int_{Y} |\nabla E_{ji}(y,\xi)- \nabla E_{ji}(y,\xi^\prime)|^2 dy
&\leq 2 \int_{Y} \big|\nabla^2 \big(f_{i}(y,\xi)-f_{i}(y,\xi^\prime)\big)\big|^2 dy\\
&\leq C \int_{2Y} \big|b_{i}(y,\xi)-b_{i}(y,\xi^\prime)\big|^2 dy
\leq C|\xi-\xi^\prime|^2
\end{aligned}
\end{equation*}
where we employ $H^2$ theory in the second step, and $\eqref{a:3}$ and $\eqref{f:2.2}$ in the last one.
This together with Poincar\'e's inequality finally leads to the
desired estimate $\eqref{pri:2.4}$, and we end the proof here.
\end{proof}

\begin{definition}\label{def:2.1}
Fix a nonnegative function $\zeta\in C_0^\infty(B(0,1/2))$, and $\int_{\mathbb{R}^d}\zeta(x)dx = 1$. Define the smoothing operator
\begin{equation}\label{def:2.1}
S_\varepsilon(f)(x) = f*\zeta_\varepsilon(x) = \int_{\mathbb{R}^d} f(x-y)\zeta_\varepsilon(y) dy,
\end{equation}
where $\zeta_\varepsilon=\varepsilon^{-d}\zeta(x/\varepsilon)$.
Let $\tilde{B}(0,1/2)\subset \mathbb{R}^{d-1}$ be a ball, and $\eta\in C^{\infty}_0(\tilde{B}(0,1/2))$ be a nonnegative function such that $\int_{\mathbb{R}^{d-1}}\eta = 1$. Then one may similarly define
\begin{equation}\label{def:2.2}
 K_\delta(g)(x) = g*\eta_\delta(x) = \int_{\mathbb{R}^{d-1}} g(x-y)\eta_\delta(y) dy,
\end{equation}
where $\eta_\delta=\delta^{-d+1}\zeta(x/\delta)$.
\end{definition}

\begin{lemma}
Let $f\in L^p(\mathbb{R}^d)$ for some $1\leq p<\infty$. Then for any $\varpi\in L_{per}^p(\mathbb{R}^d)$,
\begin{equation}\label{pri:2.6}
\big\|\varpi(\cdot/\varepsilon)S_\varepsilon(f)\big\|_{L^p(\mathbb{R}^d)}
\leq C\big\|\varpi\big\|_{L^p(Y)}\big\|f\big\|_{L^p(\mathbb{R}^d)},
\end{equation}
where $C$ depends on $\zeta$ and $d$.
\end{lemma}

\begin{proof}
See \cite[Lemma 2.1]{S5}.
\end{proof}

\begin{lemma}
Let $f\in W^{1,p}(\mathbb{R}^d)$ for some $1<p<\infty$. Then we have
\begin{equation}\label{pri:2.7}
\big\|S_\varepsilon(f)-f\big\|_{L^p(\mathbb{R}^d)}
\leq C\varepsilon\big\|\nabla f\big\|_{L^p(\mathbb{R}^d)},
\end{equation}
where $C$ depends only on $d$.
\end{lemma}

\begin{proof}
See \cite[Lemma 2.2]{S5}.
\end{proof}

\begin{lemma}
Let $g\in H^1(\mathbb{R}^{d-1})$, and then we have
\begin{equation}\label{pri:2.8}
\|K_\delta(g)\|_{H^{3/2}(\mathbb{R}^{d-1})}
\leq C\delta^{-1/2}\|g\|_{H^1(\mathbb{R}^{d-1})}
\end{equation}
and
\begin{equation}\label{pri:2.9}
\|K_\delta(g) - g\|_{H^{1/2}(\mathbb{R}^{d-1})}
\leq C\delta^{1/2}\|g\|_{H^1(\mathbb{R}^{d-1})},
\end{equation}
where $C$ depends on $d$ and $\eta$.
\end{lemma}

\begin{proof}
The main idea has been in \cite[Lemma 2.2]{S5}, and we provide a proof for
the sake of the completeness. By the definition, one may have
\begin{equation*}
\|K_\delta(g)\|_{H^{3/2}(\mathbb{R}^{d-1})}^2
= \int_{\mathbb{R}^{d-1}} (1+|\xi|^2)^{\frac{3}{2}}
|\widehat{\eta}(\delta\xi)|^2|\widehat{g}|^2 d\xi,
\end{equation*}
in which the notation ``$\widehat{g}$'' represents the Fourier transformation of $g$.
It suffices to show
\begin{equation*}
\begin{aligned}
\int_{\mathbb{R}^{d-1}}|\xi|^3|\widehat{\eta}(\delta\xi)|^2|\widehat{g}|^2 d\xi
&\leq C \int_{\mathbb{R}^{d-1}}|\xi|^2|\widehat{\eta}(\delta\xi)||\widehat{\nabla\eta}(\delta\xi)|
|\widehat{g}|^2 \frac{d\xi}{\delta} \\
&\leq C\int_{\mathbb{R}^{d-1}}|\xi|^2|\widehat{g}|^2 d\xi,
\end{aligned}
\end{equation*}
where $C$ depends on $d$ and $\eta$. Here we use the fact that
$\widehat{\nabla\eta}(\delta\xi) = 2\pi i \delta\xi \widehat{\eta}(\delta\xi)$
and
\begin{equation*}
\|\widehat{\eta}\|_{L^\infty(\mathbb{R}^{d-1})}
+ \|\widehat{\nabla\eta}\|_{L^\infty(\mathbb{R}^{d-1})} \leq C(d,\eta).
\end{equation*}

To obtain $\eqref{pri:2.9}$, it suffices to derive
\begin{equation*}
\begin{aligned}
\int_{\mathbb{R}^{d-1}} |\xi||\widehat{g}-\widehat{\eta}(\delta\xi)\widehat{g}|^2 d\xi
&\leq C\delta^2\int_{\delta|\xi|\leq 1} |\xi|^3|\widehat{g}|^2 d\xi
+ C \int_{\delta|\xi|> 1} |\xi||\widehat{g}|^2 d\xi \\
&\leq C\delta \int_{\mathbb{R}^{d-1}} |\xi|^2|\widehat{g}|^2 d\xi,
\end{aligned}
\end{equation*}
where we notice that
$\widehat{\eta}(0) = 1$ and $\widehat{\eta}\in W^{1,\infty}(\mathbb{R}^{d-1})$ in the
first inequality,  and we are done.
\end{proof}

\begin{remark}
\emph{If $g\in H^1(\partial B(0,r))$ for any $r>0$,
then there exists $(g)_\delta\in H^{3/2}(\partial B(0,r))$ such that
\begin{equation}\label{pri:2.10}
\|(g)_\delta\|_{H^{3/2}(\partial B(0,r))}
\leq C\delta^{-1/2}\|g\|_{H^1(\partial B(0,r))}
\qquad
\|(g)_\delta-g\|_{H^{1/2}(\partial B(0,r))}
\leq C\delta^{1/2}\|g\|_{H^1(\partial B(0,r))},
\end{equation}
in which the constant $C$ is independent of $r$.
This estimate is based upon the above results $\eqref{pri:2.8}$ and $\eqref{pri:2.9}$,
respectively. We mention that it has already been given in \cite{SZ1}
without a proof. Similarly, inspired by the estimate $\eqref{pri:2.7}$ we may have
\begin{equation}\label{pri:2.13}
\|(g)_\delta-g\|_{L^{2}(\partial B(0,r))}
\leq C\delta\|g\|_{H^1(\partial B(0,r))}.
\end{equation}}
\end{remark}

\begin{lemma}[interior Caccioppoli's inequality]\label{lemma:2.3}
Assume that $\mathcal{L}_\varepsilon$ satisfies the conditions
$\eqref{a:1}$, $\eqref{a:2}$, and $\eqref{a:3}$.
Let $u_\varepsilon\in H^1(B(0,2r))$ be a weak solution of
$\mathcal{L}_\varepsilon u_\varepsilon = F$ in $B(0,2r)$ with $r>0$. Then we have
\begin{equation}\label{pri:2.14}
\int_{B(0,r)}|\nabla u_\varepsilon|^2 dx
\leq \frac{C}{r^2}\inf_{c\in\mathbb{R}}\int_{B(0,2r)}|u_\varepsilon - c|^2 dx
+ Cr^2\int_{B(0,2r)}|F|^2 dx,
\end{equation}
where $C$ depends on $\mu_0,\mu_2$ and $d$.
\end{lemma}

\begin{proof}
The proof may be found in \cite{GM}.
By the definition of the weak solution,
\begin{equation}\label{pde:2.4}
\int_{B(0,2r)}A(x/\varepsilon,\nabla u_\varepsilon)\cdot\nabla\phi dx
= \int_{B(0,2r)} F\phi dx
\end{equation}
holds for any $\phi\in H_0^1(B(0,2r))$.
Set $\phi = \psi_{r}^2(u_\varepsilon - c)$ for any $c\in\mathbb{R}$, where
$\psi_{r}\in C_0^1(B(0,2r))$ is a cut-off function, satisfying
$\psi_r = 1$ in $B(0,r)$ and $\psi_r = 0$ outside $B(0,3r/2)$ with
$|\nabla\psi_r|\leq C/r$. The stated estimate $\eqref{pri:2.14}$ follows
from the assumptions $\eqref{a:1}$, $\eqref{a:2}$ and $\eqref{a:3}$ coupled
with Young's inequality, and the details will also be found in the later discussion.
\end{proof}

\begin{remark}
\emph{Let $c = \dashint_{B(0,2r)}u_\varepsilon$ in the estimate $\eqref{pri:2.14}$, and
it follows from the Sobolev-Poincar\'e inequality that
\begin{equation}\label{pri:2.17}
  \dashint_{B(0,r)} |\nabla u_\varepsilon|^2
  \leq C\Big(\dashint_{B(0,2r)}|\nabla u_\varepsilon|^{\frac{2d}{d+2}}\Big)^{\frac{d+2}{d}}
  + Cr^2\dashint_{B(0,2r)}|F|^2.
\end{equation}
Using the reverse inequality (see \cite[Chapter V, Theorem 1.2]{GM}),
we can obtain higher integrability, and there exists $p_0>2$, depending on $\mu_0,\mu_2,d$, such that
\begin{equation}\label{pri:2.16}
 \dashint_{B(0,r)} |\nabla u_\varepsilon|^p
 \leq C\Big(\dashint_{B(0,2r)} |\nabla u_\varepsilon|^2\Big)^{p/2}
 + Cr^p\dashint_{B(0,2r)}|F|^p
\end{equation}
holds for any $2\leq p< p_0$,
where $C$ depends on $\mu_0,\mu_2,d$ and $p$.}
\end{remark}

\begin{thm}[$H^1$ theory]\label{thm:2.1}
Let $\Omega$ be a bounded Lipschitz domain.
Assume that $\mathcal{L}_\varepsilon$ satisfies the conditions
$\eqref{a:1}$, $\eqref{a:2}$, and $\eqref{a:3}$. Let
$u_\varepsilon\in H^1(\Omega)$ be the solution of $\eqref{pde:1.1}$. Then we have
\begin{equation}\label{pri:2.12}
 \|\nabla u_\varepsilon\|_{L^2(\Omega)}
 \leq C\Big\{r_0\|F\|_{L^2(\Omega)} + \|g\|_{H^{1/2}(\partial\Omega)}\Big\},
\end{equation}
where $C$ depends on $\mu_0,\mu_2,d$ and the character of $\Omega$.
\end{thm}

\begin{proof}
The proof is standard, and we provide a proof for the sake of the completeness.
By the definition of the weak solution,
one may have
\begin{equation*}
\int_{\Omega}A(x/\varepsilon,\nabla u_\varepsilon)\cdot\nabla(u_\varepsilon - z) dx
= \int_{\Omega} F(u_\varepsilon - z) dx,
\end{equation*}
where
$z\in H^1(\Omega)$ satisfies
\begin{equation*}
 \Delta z = 0 \quad \text{in}~~\Omega,
 \qquad z = g \quad \text{on}~~\partial\Omega.
\end{equation*}
It is well known that
\begin{equation}\label{f:2.10}
  \|\nabla z\|_{L^2(\Omega)}\leq C\|g\|_{H^{1/2}(\partial\Omega)},
\end{equation}
where $C$ depends on $d$ and the character of $\Omega$. Then by the assumptions
$\eqref{a:1}$, $\eqref{a:2}$ and $\eqref{a:3}$ we have
\begin{equation*}
\begin{aligned}
&\int_{\Omega} A(x/\varepsilon,\nabla u_\varepsilon)
\cdot(\nabla u_\varepsilon - \nabla z)dx \\
&\geq \mu_0\int_{\Omega}|\nabla u_\varepsilon|^2 dx
-\frac{\mu_0}{2}\int_{\Omega}|\nabla u_\varepsilon|^2 dx
-C(\mu_0,\mu_2)\int_{\Omega}|\nabla z|^2 dx,
\end{aligned}
\end{equation*}
in which we use Young's inequality, and
\begin{equation*}
\begin{aligned}
\Big|\int_{\Omega}F(u_\varepsilon - z)dx\Big|
&\leq Cr_0\|F\|_{L^2(\Omega)}\|\nabla(u_\varepsilon - z)\|_{L^2(\Omega)}\\
&\leq \frac{\mu_0}{4}\int_{\Omega}|\nabla u_\varepsilon|^2 dx
+ Cr_0^2\int_{\Omega}|F|^2 dx + C\int_{\Omega}|\nabla z|^2 dx
\end{aligned}
\end{equation*}
where we employ Poincar\'e's inequality in the first step. Thus we have
\begin{equation}\label{f:2.11}
\int_{\Omega}|\nabla u_\varepsilon|^2 dx
\leq Cr_0^2\int_{\Omega}|F|^2 dx + C\int_{\Omega}|\nabla z|^2 dx
\end{equation}
Consequently, this together with the estimate $\eqref{f:2.10}$ leads to
the desired estimate $\eqref{pri:2.12}$, and we are done.
\end{proof}

\begin{thm}[$L^p$ estimates]\label{thm:2.2}
Assume the same conditions as in Theorem $\ref{thm:2.1}$.
Given $F\in L^p(\Omega)$ for some $p>2$ and $g\in W^{1-1/p,p}(\partial\Omega)$,
let $u_\varepsilon\in H^1(\Omega)$ be the solution of $\eqref{pde:1.1}$ and
$\eqref{pdeD:1}$. Then we have the following uniform estimate
\begin{equation}\label{pri:2.12}
 \|\nabla u_\varepsilon\|_{L^p(\Omega)}
 \leq C_p\Big\{r_0\|F\|_{L^p(\Omega)} + \|g\|_{W^{1-1/p,p}(\partial\Omega)}\Big\},
\end{equation}
where $C_p$ depends on $\mu_0,\mu_2,d,p$ and the character of $\Omega$.
\end{thm}

\begin{lemma}[reverse H\"older inequality]\label{lemma:2.2}
Assume the same conditions as in Theorem $\ref{thm:2.1}$. Then, there exist
positive constants $\delta_0, R_0$ and $C$, depending on $\mu_0,\mu_2,d$ and the
character of $\Omega$, such that
$|\nabla u_\varepsilon|\in L^{p}_{loc}(\Omega)$ with $p=2(1+\delta_0)$
and, we have the following estimate
\begin{equation}\label{pri:2.15}
  \dashint_{\Omega(x,r)}|\nabla u_\varepsilon|^{2(1+\delta)}
  \leq
  C\Bigg\{
  \Big(\dashint_{\Omega(x,2r)}|\nabla u_\varepsilon|^2 \Big)^{1+\delta}
  + \dashint_{\Omega(x,2r)}|rF|^{2(1+\delta)}
  + \dashint_{B(x,2r)}|\nabla\tilde{g}|^{2(1+\delta)}\Bigg\}
\end{equation}
for any $\delta\in(0,\delta_0]$, $0<r<(R_0/4)$ and $x\in\overline{\Omega}$,
where $\Omega(x,r)=B(x,r)\cap\Omega$, and
$\tilde{g}$ is an extension of $g$ to $\mathbb{R}^d$ in the sense of $W^{1,p}$
norm.
\end{lemma}

\begin{proof}
The proof is based upon reverse H\"older inequality
(see for example \cite[Chapter V, Theorem 1.2]{GM}), and the main idea
may be found in \cite[Lemma 4.1]{F}.
Since the reverse H\"older type inequalities involved
are verified only on cubes or balls, let $u_\varepsilon = \tilde{g}$
and $F = 0$
in $\mathbb{R}^d\setminus\overline{\Omega}$  for the convenience.
Fixed a ball $B(x_0,R_0/2)\in\mathbb{R}^d$ with $x_0\in\partial\Omega$,
for any $x\in B(x_0,R_0/2)$ and
$B(x,r)\subset B(x_0,R_0/4)$,  we plan to establish
\begin{equation}\label{f:2.13}
\dashint_{B(x,r)}|\nabla u_\varepsilon - \nabla \tilde{g}|^2
\leq C\Big(\dashint_{B(x,2r)}
|\nabla u_\varepsilon - \nabla \tilde{g}|^{\frac{2d}{d+2}}\Big)^{\frac{d+2}{d}}
+C\dashint_{B(x,2r)}\big(r^2|F|^2+|\nabla\tilde{g}|^2\big),
\end{equation}
where $C$ depends on $\mu_0,\mu_2$ and $d$. Define
\begin{equation}\label{def:2.4}
\tilde{F}(y) = C|F(y)|\text{dist}(y,M_{3r}^{4r}(x))
+ C|\nabla\tilde{g}(y)|
\end{equation}
for $y\in B(x,2r)$, and $M_{3r}^{4r}(x) = B(x,4r)\setminus B(x,3r)$.
Thus, it follows from \cite[Chapter V, Theorem 1.2]{GM} that there exits
a small constant $\delta_0>0$ such that, for any $\delta\in(0,\delta_0]$,
\begin{equation*}
\begin{aligned}
\dashint_{B(x,r)}|\nabla u_\varepsilon - \nabla \tilde{g}|^{2(1+\delta)}
&\leq C\Big(\dashint_{B(x,2r)}
|\nabla u_\varepsilon - \nabla \tilde{g}|^2\Big)^{1+\delta}
+C\dashint_{B(x,2r)}|\tilde{F}|^{(1+\delta)}\\
&\leq C\Big(\dashint_{B(x,2r)}
|\nabla u_\varepsilon - \nabla \tilde{g}|^2\Big)^{1+\delta}
+C\dashint_{B(x,2r)}\Big(|rF|^{2(1+\delta)}
+|\nabla\tilde{g}|^{2(1+\delta)}\Big).
\end{aligned}
\end{equation*}
By noting that there exists a constant $c_0\in(0,1)$, such that
$c_0|B(x,r)|\leq |\Omega(x,r)|\leq |B(x,r)|$ for any $x\in\overline{\Omega}$ and
$0<r<(R_0/4)$, a routine computation leads to the desired estimate
$\eqref{pri:2.15}$. The proof is reduced to show the estimate $\eqref{f:2.13}$.
In the case of $B(x,r)\subset\mathbb{R}^d\setminus\overline{\Omega}$, the left-hand
side of $\eqref{f:2.13}$ will vanish and there is nothing to prove.

We now turn to the case $B(x,3r/2)\subset\Omega$, and it follows from
the interior estimate $\eqref{pri:2.17}$ that
\begin{equation*}
\begin{aligned}
\dashint_{B(x,r)}|\nabla u_\varepsilon - \nabla\tilde{g}|^2
&\leq 2\dashint_{B(x,r)}|\nabla u_\varepsilon|^2
+ 2\dashint_{B(x,r)}|\nabla\tilde{g}|^2 \\
&\leq C\Big(\dashint_{B(x,2r)}|\nabla u_\varepsilon
-\nabla\tilde{g}|^{\frac{2d}{d+2}}\Big)^{\frac{d+2}{d}}
+ Cr^2\dashint_{B(x,2r)}|F|^2 + C\dashint_{B(x,2r)}|\nabla\tilde{g}|^2,
\end{aligned}
\end{equation*}
where we also use H\"older's inequality in the last step. The third case is
$B(x.3r/2)\cap\partial\Omega \not=\varnothing$. Rewriting $\eqref{pde:2.4}$ as
\begin{equation}\label{pde:2.5}
\int_{\Omega(x,2r)}A(x/\varepsilon,\nabla u_\varepsilon)\cdot\nabla\phi dx
= \int_{\Omega(x,2r)} F\phi dx
\end{equation}
for any $\phi\in H_0^1(\Omega(x,2r))$,
we consider put $\phi=\psi_{r}^2(u_\varepsilon-\tilde{g})$
into the above equation, where $\phi_r\in C_0^1(B(x,2r))$ is a cut-off function.
\begin{equation*}
\begin{aligned}
\text{LHS~of~}  \eqref{pde:2.5}~ &= \int_{\Omega(x,2r)}\psi_r^2 A(z/\varepsilon,\nabla u_\varepsilon)
\cdot \nabla(u_\varepsilon-\tilde{g}) dz
+ 2\int_{\Omega(x,2r)}\psi_r A(z/\varepsilon,\nabla u_\varepsilon)
\nabla\psi_{r}(u_\varepsilon-\tilde{g}) dz\\
&\geq \frac{\mu_0}{2}\int_{\Omega(x,2r)}
\psi_r^2 |\nabla u_\varepsilon - \nabla\tilde{g}|^2 dz
-\frac{C}{r^2}\int_{\Omega(x,2r)}|u_\varepsilon - \tilde{g}|^2 dz
-C\int_{B(x,2r)}|\nabla\tilde{g}|^2 dz,
\end{aligned}
\end{equation*}
where we use the assumptions $\eqref{a:1}$, $\eqref{a:2}$ and $\eqref{a:3}$,
as well as  Young's inequality.  Moreover, we have
\begin{equation}\label{f:2.14}
 \text{LHS}
 \geq \frac{\mu_0}{2}\int_{\Omega(x,r)}
 |\nabla u_\varepsilon - \nabla\tilde{g}|^2 dz
-\frac{C}{r^2}\Big(\int_{\Omega(x,2r)}|\nabla u_\varepsilon -
\nabla\tilde{g}|^{\frac{2d}{d+2}} dz\Big)^{\frac{d+2}{d}}
-C\int_{B(x,2r)}|\nabla\tilde{g}|^2 dz,
\end{equation}
where we use the Sobolev-Poincar\'e inequality in the computation.
Similarly, the right-hand side of $\eqref{pde:2.5}$ is controlled  by
\begin{equation*}
\frac{\mu_0}{4}\int_{\Omega(x,r)}
 |\nabla u_\varepsilon - \nabla\tilde{g}|^2 dz
+ Cr^2\int_{\Omega(x,2r)}|F|^2 dz
+ \frac{C}{r^2}\Big(\int_{\Omega(x,2r)}|\nabla u_\varepsilon -
\nabla\tilde{g}|^{\frac{2d}{d+2}} dz\Big)^{\frac{d+2}{d}}.
\end{equation*}

Finally, this together with $\eqref{f:2.14}$ implies the desired estimate
$\eqref{f:2.13}$, in which we note that
$|\nabla u_\varepsilon - \nabla \tilde{g}|=0$ on $B(x,r)
\setminus\overline{\Omega}$. We have completed the whole proof.
\end{proof}

\noindent\textbf{Proof of Theorem $\ref{thm:2.2}$}.
The proof is based upon Lemma $\ref{lemma:2.2}$.
Note that the estimate $\eqref{pri:2.13}$ is in fact unified
the interior estimate $\eqref{pri:2.16}$ and the related boundary estimate. Let
$\delta,R_0$ be given as in Lemma $\ref{lemma:2.2}$, and we have
\begin{equation*}
\int_{\Omega(x,r)}|\nabla u_\varepsilon|^{2(1+\delta)}
\leq Cr^{-d\delta}\Big(\int_{\Omega(x,2r)}|\nabla u_\varepsilon|^2\Big)^{1+\delta}
+ C\int_{\Omega(x,2r)}|rF|^{2(1+\delta)}
+ C\int_{\Omega(x,2r)}|\nabla\tilde{g}|^{2(1+\delta)}
\end{equation*}
for any $x\in\overline{\Omega}$ and $0<r<(R_0/4)$. On account of a covering argument, the
above estimate implies
\begin{equation*}
\begin{aligned}
\int_{\Omega}|\nabla u_\varepsilon|^{2(1+\delta)}
&\leq Cr^{-d\delta}\Big(\int_{\Omega}|\nabla u_\varepsilon|^2\Big)^{1+\delta}
+ C\int_{\Omega}|rF|^{2(1+\delta)}
+ C\int_{\Omega}|\nabla\tilde{g}|^{2(1+\delta)}\\
&\leq C(r_0/r)^{d\delta}r_0^{2(1+\delta)}\int_{\Omega}|F|^{2(1+\delta)}
+ C\big[(r_0/r)^{d\delta}+1\big]\int_{\Omega}|\nabla\tilde{g}|^{2(1+\delta)},
\end{aligned}
\end{equation*}
where we employ H\"older's inequality and the estimate $\eqref{f:2.11}$ in the second step.
In fact, we may choose $r\in(0,R_0/4)$ such that the radius $r$ and $r_0$ are comparable, and
by setting $p=2(1+\delta)$, we will obtain the stated estimate $\eqref{pri:2.12}$, in which
the estimate $\|\tilde{g}\|_{W^{1,p}(\tilde{\Omega})}
\leq C\|g\|_{W^{1-1/p,p}(\partial\Omega)}$ is used and $\tilde{\Omega}\supseteq\Omega$.
We have completed the proof.
\qed

\begin{remark}
\emph{In Theorems $\ref{thm:2.1}$, $\ref{thm:2.2}$,
we do not use the periodicity condition in the proof.}
\end{remark}

\begin{thm}[$H^2$ theory]\label{thm:2.3}
Let $\Omega$ be a bounded $C^{1,1}$ domain.
Given $g\in H^{3/2}(\partial\Omega)$ and $F\in L^2(\Omega)$,
assume that
$u_0\in H^1(\Omega)$ is the weak solution of
$\mathcal{L}_0 u_0 = F$ in $\Omega$ with $u_0 = g$ on $\partial\Omega$.
Then we have $u_0\in H^2(\Omega)$ satisfying
\begin{equation}\label{pri:2.11}
 \|\nabla^2 u_0\|_{L^2(\Omega)}
 \leq C\Big\{\|F\|_{L^2(\Omega)}+\|g\|_{H^{3/2}(\partial\Omega)}\Big\}
\end{equation}
where $C$ depends on $\mu_0,\mu_2,d$ and the character of $\Omega$.
\end{thm}

\begin{proof}
The main idea of the proof is standard,
and we provide a proof for the sake of the completeness.
The interior $H^2$ estimate follows from \cite[pp.46]{GM},
while we focus on the boundary estimate and first
study the special case $\Omega = B^{+}(0,R)=B(0,R)\cap\mathbb{R}^d_{+}$ with $R>0$, and $g=0$.
For the ease of the statement, it is fine to assume $u_0\in H^2(B^+(0,R))$,  and we have
\begin{equation}\label{pde:2.1}
\left\{\begin{aligned}
&\int_{B^{+}(0,R)}\nabla_{\xi_j}A^i(\nabla u_0)\nabla^2_{jk}u_0\nabla_i\phi dx &=
-\int_{B^{+}(0,R)} F \nabla_k\phi dx,  ~~\\
&\quad\nabla_k u_0 = 0 \qquad \text{on}~~ T(0,R)
\end{aligned}\right.
\end{equation}
for any $\varphi\in H^1_0(B^+(0,R))$,
where $k=1,2,\cdots,d-1$, and $T(0,R) = B(0,R)\cap\{x\in\mathbb{R}^d:x_d=0\}$.
Let $\phi = \psi^2 \nabla_k u_0$, where
$\psi\in C^1_0(B(0,R))$ is a cut off function satisfying
$\psi = 1$ in $B(0,R/2)$ and $\psi=0$ outside B(0,2R/3) with
$|\nabla\psi|\leq C/R$. Then it follows from the condition $\eqref{a:4}$ that
\begin{equation*}
\int_{B^+(0,R/2)}|\nabla \nabla_k u_0|^2 dx
\leq \frac{C}{R^2}\int_{B^{+}(0,R)}|\nabla u_0|^2 dx
+ C\int_{B^{+}(0,R)}|F|^2 dx.
\end{equation*}
On the other hand, we observe that
\begin{equation}\label{pde:2.2}
- \nabla_{\xi_j}A^{i}(\nabla u_0)\nabla^2_{ij} u_0 = F
 \qquad \text{in}~~B^{+}(0,R),
\end{equation}
which implies
\begin{equation*}
  |\nabla_{dd}^2 u_0| \leq \frac{1}{\mu_0}\Big\{|F|
  + C\sum_{1\leq i\leq d\atop 1\leq j\leq d-1}|\nabla_{ij}^2 u_0|\Big\},
\end{equation*}
where we use the condition $\eqref{pri:2.3}$. By the above formula, we have
\begin{equation*}
\int_{B^+(0,R/2)}|\nabla^2 u_0|^2 dx
\leq \frac{C}{R^2}\int_{B^{+}(0,R)}|\nabla u_0|^2 dx
+ C\int_{B^{+}(0,R)}|F|^2 dx.
\end{equation*}

Next, we handle the case of $g\not=0$. To do so, we construct a Lipschitz domain
$\tilde{\Omega}$ such that $T(0,2R)\subset\partial\tilde{\Omega}$
and $B^+(0,R)\subset \tilde{\Omega}$,
and an extension of $\nabla_kg$ in the sense of $H^{1/2}$ norm, denoted by
$\widetilde{\nabla_k g}$, such that
$\|\widetilde{\nabla_k g}\|_{H^1(\partial\tilde{\Omega})}\leq
C\|\nabla_k g\|_{H^{1/2}(T(0,T))}$, where $k=1,\cdots, d-1$. Then there exists
$Z\in H^1(\tilde{\Omega})$ satisfying $\Delta z = 0$ in $\tilde{\Omega}$ with
$Z= \widetilde{\nabla_k g}$ on $\partial\tilde{\Omega}$, and one may derive that
\begin{equation}\label{f:2.7}
\|\nabla Z\|_{L^2(B^{+}(0,R))}
\leq C\|\nabla_k g\|_{H^{1/2}(T(0,R))},
\end{equation}
where $C$ is independent of $R$. Let $w = \nabla_k u_0 -z$, and in view of
$\eqref{pde:2.1}$ we have
\begin{equation}\label{pde:2.3}
\left\{\begin{aligned}
-\text{div}(\nabla_{\xi_j}A(\nabla u_0)\nabla_jw)
&= \text{div}(\tilde{F}+\nabla_{\xi_j}A(\nabla u_0)\nabla_jZ)
&\quad&\text{in}~~ B^+(0,R), \\
w &=0 &\quad&\text{on}~~T(0,R),
\end{aligned}\right.
\end{equation}
where $\tilde{F} = Fe_k$.
Then it follows from Caccioppoli's inequality $\eqref{pri:5.1}$ that
\begin{equation*}
\int_{B^+(0,R/2)}|\nabla \nabla_k u_0|^2dx
\leq \frac{C}{R^2}\int_{B^+(0,R)}|\nabla u_0|^2 dx
+ C \Big\{\int_{B^+(0,R)}|F|^2 dx + \|g\|^2_{H^{3/2}(T(0,R))}\Big\},
\end{equation*}
and this together with $\eqref{pde:2.2}$ gives
\begin{equation}\label{f:2.5}
\|\nabla^2 u_0\|_{L^2(B^{+}(0,R/2))}
\leq CR^{-1}\|\nabla u_0\|_{L^2(B^{+}(0,R))}
+C\Big\{\|F\|_{L^2(B^{+}(0,R))} + \|g\|_{H^{3/2}(T(0,R))}\Big\},
\end{equation}
where $C$ is independent of $R$. Now we have
figured out the right space that the boundary data $g$ belongs to.
The remainder of the proof is to employ
the so-called straightening the boundary arguments to handle
the case of a general $C^{1,1}$ domain, and
proceeding the proof is too complicated to be given here.
We have completed the proof.
\end{proof}

\section{Convergence rates}

\begin{lemma}\label{lemma:3.1}
Let $\Omega\subset\mathbb{R}^d$ be a bounded Lipschitz domain.
Suppose that $u_\varepsilon, u_0\in H^1(\Omega)$
satisfy $\mathcal{L}_\varepsilon u_\varepsilon = \mathcal{L}_0 u_0$ in
$\Omega$. For any $\varphi\in H^1_0(\Omega;\mathbb{R}^d)$,
the first-order approximating corrector is given by
\begin{equation}\label{eq:3.1}
 v_\varepsilon(x) = u_0(x) + \varepsilon N(x/\varepsilon,\varphi).
\end{equation}
Then for any $\phi\in H_0^1(\Omega)$ we have
\begin{equation}\label{pri:2.5}
\begin{aligned}
&\int_{\Omega} \big(A(x/\varepsilon,\nabla u_\varepsilon) - A(x/\varepsilon,\nabla v_\varepsilon)\big)
\cdot \nabla\phi dx\\
&\leq C\int_\Omega\Big(
|\nabla u_0 - \varphi| + \varepsilon|\nabla_\xi N(x/\varepsilon,\varphi)\nabla\varphi|
+\varepsilon|\nabla_\xi E(x/\varepsilon,\varphi)\nabla \varphi|\Big)|\nabla\phi| dx,
\end{aligned}
\end{equation}
where $C$ depends only on $\mu_2$ and $d$.
\end{lemma}

\begin{proof}
Our proof is inspired by \cite{P}, and in view of
$\mathcal{L}_\varepsilon u_\varepsilon = \mathcal{L}_0 u_0$ in $\Omega$,
the left-hand side of $\eqref{pri:2.5}$ is equal to
\begin{equation*}
\begin{aligned}
\int_\Omega \big(A(x/\varepsilon,\nabla u_\varepsilon) - A(x/\varepsilon,\nabla v_\varepsilon)\big)\cdot
\nabla\phi dx
& = \int_\Omega \Big(\widehat{A}(\nabla u_0) - \widehat{A}(\varphi) \\
&+ \widehat{A}(\varphi) - A(y,\varphi+\nabla_y N(y,\varphi)) \\
&+ A(y,\varphi+\nabla_y N(y,\varphi)) - A(y,\nabla v_\varepsilon)\Big)\cdot\nabla\phi dx,
\end{aligned}
\end{equation*}
where $y=x/\varepsilon$.
Then we calculate the right-hand side above term by term. On account of $\eqref{pri:2.3}$,
\begin{equation}\label{f:2.3}
|\widehat{A}(\nabla u_0) - \widehat{A}(\varphi)|\leq C|\nabla u_0 - \varphi|.
\end{equation}
By $\eqref{a:3}$, we have
\begin{equation}\label{f:2.4}
\big|A(y,\varphi+\nabla_y N(y,\varphi)) - A(y,\nabla v_\varepsilon)\big|
\leq \mu_2\Big(|\nabla u_0 - \varphi|+\varepsilon|\nabla_\xi N(y,\varphi)\nabla\varphi|\Big).
\end{equation}
Recalling that $b(y,\varphi) = \widehat{A}(\varphi) - A(y,\varphi+\nabla_y N(y,\varphi))$, it follows
from $\eqref{eq:2.1}$ that
\begin{equation*}
\begin{aligned}
\int_{\Omega} b(x/\varepsilon,\varphi)\cdot\nabla\phi dx
&= \varepsilon\int_{\Omega}\frac{\partial}{\partial x_j}\big\{E_{ji}(x/\varepsilon,\varphi)\big\}
\frac{\partial\phi}{\partial x_i} dx
- \varepsilon\int_{\Omega}\frac{\partial E_{ji}}{\partial \xi_k}\frac{\partial \varphi_k}{\partial x_j}
\frac{\partial\phi}{\partial x_i} dx\\
&= - \varepsilon\int_{\Omega}\frac{\partial E_{ji}}{\partial \xi_k}\frac{\partial \varphi_k}{\partial x_j}
\frac{\partial\phi}{\partial x_i}dx
\leq \varepsilon\int_\Omega |\nabla_{\xi}E(y,\varphi)\nabla \varphi||\nabla\phi| dx
\end{aligned}
\end{equation*}
This together with $\eqref{f:2.3}$ and $\eqref{f:2.4}$ gives the stated estimate $\eqref{pri:2.5}$ and
we are done.
\end{proof}

We impose the following cut-off function $\psi_{r}\in C_0^1(\Omega)$
associated with $\Sigma_r$:
\begin{equation}\label{def:3.1}
 \psi_r = 1 \quad\text{in}~\Sigma_{2r},
 \qquad
 \psi_r = 0 \quad\text{outside}~\Sigma_{2r},
 \qquad
 |\nabla \psi_r|\leq C/r,
\end{equation}
where we recall the notation $\Sigma_r = \{x\in\Omega:\text{dist}(x,\partial\Omega)>r\}$.

\begin{lemma}
Assume the same conditions as in Lemma $\ref{lemma:3.1}$,
and $u_\varepsilon = u_0$ on $\partial\Omega$.  Let $\varphi =
S_\varepsilon(\psi_{4\varepsilon}\nabla u_0)$ in $\eqref{eq:3.1}$, and
$w_\varepsilon = u_\varepsilon - v_\varepsilon$. Then we have the following estimate
\begin{equation}\label{pri:3.1}
\|\nabla w_\varepsilon\|_{L^2(\Omega)} \leq
C\Big\{\|\nabla u_0\|_{L^2(\Omega\setminus\Sigma_{4\varepsilon})}
+ \varepsilon\|\nabla^2 u_0\|_{L^2(\Sigma_{2\varepsilon})}\Big\}
\end{equation}
and
\begin{equation}\label{pri:3.2}
\|u_\varepsilon - u_0\|_{L^2(\Omega)}
\leq Cr_0\Big\{\|\nabla u_0\|_{L^2(\Omega\setminus\Sigma_{4\varepsilon})}
+ \varepsilon\|\nabla^2 u_0\|_{L^2(\Sigma_{2\varepsilon})}\Big\},
\end{equation}
where $r_0 = \emph{diam}(\Omega)$, and $C$ depends only on $\mu_0,\mu_2$ and $d$.
\end{lemma}

\begin{proof}
We first claim that $w_\varepsilon$ vanishes on $\partial\Omega$
in the sense of the trace. In view of Remark $\ref{remark:2.1}$, one may have
$N(x/\varepsilon,S_\varepsilon(\psi_{2\varepsilon}\nabla u_0)(x)) = 0$
for any $x\in \mathbb{R}^d\setminus\Sigma_{\varepsilon}$. This coupled with
$u_\varepsilon = u_0$ on $\partial\Omega$ leads to the fact $w_\varepsilon\in H_0^1(\Omega)$.
Thus, the left-hand side of $\eqref{pri:2.5}$ and the assumption $\eqref{a:3}$ give
\begin{equation*}
\begin{aligned}
\|\nabla w_\varepsilon\|_{L^2(\Omega)}^2
&\leq C\|\nabla w_\varepsilon\|_{L^2(\Omega)}
\Big\{\|\nabla u_0 - S_\varepsilon(\psi_{2\varepsilon}\nabla u_0)\|_{L^2(\Omega)}
+ \varepsilon\|\varpi(\cdot/\varepsilon)\nabla S_\varepsilon(\psi_{2\varepsilon}\nabla u_0)\|_{L^2(\Omega)}
\Big\} \\
&\leq C\|\nabla w_\varepsilon\|_{L^2(\Omega)}
\Big\{
\|(1-\psi_{2\varepsilon})\nabla u_0\|_{L^2(\Omega)}
+\|\psi_{2\varepsilon}\nabla u_0 - S_\varepsilon(\psi_{2\varepsilon}\nabla u_0)\|_{L^2(\mathbb{R}^d)} \\
& + \varepsilon\|\varpi(\cdot/\varepsilon)
\nabla S_\varepsilon(\psi_{2\varepsilon}\nabla u_0)\|_{L^2(\mathbb{R}^d)}
\Big\}
\end{aligned}
\end{equation*}
where
$\varpi(\cdot/\varepsilon) = \nabla_\xi N(\cdot/\varepsilon,\varphi)$
or $\nabla_\xi E(\cdot/\varepsilon,\varphi)$.
Moreover, it follows from the estimates $\eqref{pri:2.2}$ and $\eqref{pri:2.4}$ that
$\|\varpi\|_{L^2(Y)} \leq C(\mu_0,\mu_2,d)$. This together with the estimates $\eqref{pri:2.6}$
and $\eqref{pri:2.7}$
consequently implies the stated estimate $\eqref{pri:3.1}$.

To show the estimate $\eqref{pri:3.2}$, it suffices to prove
\begin{equation}\label{f:3.1}
 \int_\Omega |N(x/\varepsilon,\varphi)|^2 dx \leq C\int_{\Omega}|\psi_{2\varepsilon}\nabla u_0|^2 dx,
\end{equation}
and we recall that $\varphi = S_\varepsilon(\psi_{2\varepsilon}\nabla u_0)$.
To do so, we collect a family of small cubes by $Y_\varepsilon^i = \varepsilon(i+Y)$ for
$i\in\mathbb{Z}^d$ with an index set $I_\varepsilon$, such that $\Sigma_{\varepsilon}\subset
\cup_{i\in I_\varepsilon} Y_\varepsilon^i \subset \Omega$
and $Y_\varepsilon^i\cap Y_\varepsilon^j = \emptyset$ if $i\not=j$. Thus
\begin{equation*}
\begin{aligned}
 \int_\Omega |N(x/\varepsilon,\varphi)|^2 dx
& \leq \sum_{i\in I_\varepsilon}\int_{Y^i_\varepsilon} |N(x/\varepsilon,\varphi)|^2 dx
 +  \int_{\Omega\setminus\Sigma_{\varepsilon}}
 |N(x/\varepsilon,\varphi)|^2 dx \\
& \leq C \sum_{i\in I_\varepsilon} |Y_\varepsilon^i||\varphi^i|^2
\leq C \int_{\Omega}|\varphi|^2 dx,
\end{aligned}
\end{equation*}
where we employ the estimate $\eqref{pri:2.1}$ and the fact that
\begin{equation*}
N(x/\varepsilon,\varphi)  = 0  \qquad \forall x\in\Omega\setminus\Sigma_{\varepsilon}
\end{equation*}
(see Remark $\ref{remark:2.1}$) in the second inequality,
Here we take $\varphi^i = \inf_{x\in Y_\varepsilon^i}
|S_\varepsilon(\psi_{2\varepsilon}\nabla u_0)(x)|$, and the last step is due to
Chebyshev's inequality. Therefore, the estimate $\eqref{f:3.1}$ consequently follows from
\begin{equation*}
\|S_\varepsilon(\psi_{2\varepsilon}\nabla u_0)\|_{L^2(\Omega)}
\leq \|\psi_{2\varepsilon}\nabla u_0\|_{L^2(\Omega)}
\leq Cr_0\Big\{\varepsilon^{-1}\|\nabla u_0\|_{L^2(\Omega\setminus\Sigma_{4\varepsilon})}
+ \|\nabla^2 u_0\|_{L^2(\Sigma_{2\varepsilon})}\Big\},
\end{equation*}
and we have completed the proof.
\end{proof}

\begin{thm}\label{thm:3.1}
Let $B=B(0,R)\subset\mathbb{R}^d$ be a ball with $R\in (16\varepsilon,1]$.
Assume $\mathcal{L}_\varepsilon$ satisfies the conditions $\eqref{a:1},\eqref{a:2}$ and $\eqref{a:3}$.
Given $F\in L^2(B)$ and $g\in H^{3/2}(\partial B)$,
let $u_\varepsilon, u_0\in H^1(B)$ be the weak solutions
of $\eqref{pde:1.1}$ and $\eqref{pde:1.3}$
with $u_\varepsilon = u_0 = g$ on $\partial B$.
Then we have
\begin{equation}\label{pri:3.3}
\|u_\varepsilon-u_0\|_{L^2(B)}
\leq C\Big(\frac{\varepsilon}{R}\Big)^{1/2}\bigg\{R^2\|F\|_{L^2(B)}
+R\|g\|_{H^{3/2}(\partial B)}\bigg\},
\end{equation}
where $C$ depends on $\mu_0,\mu_2, d$, but independent of $R$.
\end{thm}

\begin{proof}
In view of the estimates $\eqref{pri:3.2}$ and $\eqref{pri:2.11}$, one may have
\begin{equation}\label{f:2.8}
\|u_\varepsilon-u_0\|_{L^2(B)}
\leq CR\|\nabla u_0\|_{L^2(B\setminus B(0,R-4\varepsilon))}
+ CR^2\Big(\frac{\varepsilon}{R}\Big)
\bigg\{\|F\|_{L^2(B)}+\|g\|_{H^{3/2}(\partial B)}\bigg\}.
\end{equation}
To complete the proof, it suffices to show
\begin{equation*}
\begin{aligned}
\|\nabla u_0\|_{L^2(B\setminus B(0,R-4\varepsilon))}^2
&= \int_{0}^{4\varepsilon}\int_{\partial B(0,r)}|\nabla u_0|^2 dS_r dr
\leq 4\varepsilon \sup_{(3R/4)\leq r\leq R}
\int_{\partial B(0,r)}|\nabla u_0|^2 dS_r\\
&\leq C\varepsilon/R \int_{B(0,R)}|\nabla u_0|^2dx
+ C\varepsilon R\int_{B(0,R)}|\nabla^2 u_0|^2 dx,
\end{aligned}
\end{equation*}
where we use the trace theorem
\begin{equation*}
\int_{\partial B(0,r)}|\nabla u_0|^2 dS_r
\leq \frac{C}{R}\int_{B(0,R)}|\nabla u_0|^2 dx
+ CR\int_{B(0,R)}|\nabla^2 u_0|^2 dx
\end{equation*}
for any $(3R/4)\leq r\leq R$. Thus,
\begin{equation}\label{f:2.9}
\begin{aligned}
\|\nabla u_0\|_{L^2(B\setminus B(0,R-4\varepsilon))}
&\leq C\Big(\frac{\varepsilon}{R}\Big)^{1/2}
\bigg\{\|\nabla u_0\|_{L^2(B(0,R))}
+ R\|\nabla^2 u_0\|_{L^2(B(0,R))}\bigg\} \\
&\leq C\Big(\frac{\varepsilon}{R}\Big)^{1/2}
\bigg\{R\|F\|_{L^2(B(0,R))}
+ \|g\|_{H^{3/2}(\partial B(0,R))}\bigg\},
\end{aligned}
\end{equation}
in which we employ the estimates $\eqref{pri:2.11}$ and $\eqref{pri:2.12}$,
and the fact that $\|g\|_{H^{1/2}(\partial B)}
\leq C\|g\|_{H^{3/2}(\partial B)}$.

Consequently, the desired estimate $\eqref{pri:3.3}$ follows from
$\eqref{f:2.8}$ and $\eqref{f:2.9}$ by noting that $\varepsilon/R\leq 1$, and
we have completed the proof.
\end{proof}

\begin{thm}\label{thm:3.2}
Let $\Omega$ be a bounded Lipschitz domain with $\varepsilon\leq r_0\leq 1$. Assume
that $u_\varepsilon, u_0\in H^1(\Omega)$ are the weak solutions
of $\eqref{pde:1.1}$ and $\eqref{pde:1.3}$
with $u_\varepsilon = u_0 = g$ on $\partial\Omega$, respectively.
Let $F\in L^p(\Omega)$ for some $p>2$ and $g\in W^{1-1/p,p}(\partial\Omega)$.
Then exists $C>0$, depending on $\mu_0,\mu_2,d,p$ and the
character of $\Omega$, such that
\begin{equation}\label{pri:3.4}
\|u_\varepsilon - u_0\|_{L^2(\Omega)}
\leq Cr_0^{1+d\sigma}
\Big(\frac{\varepsilon}{r_0}\Big)^\sigma\bigg\{r_0\|F\|_{L^p(\Omega)}
+\|g\|_{W^{1-1/p,p}(\partial\Omega)}\bigg\},
\end{equation}
where $\sigma = 1/2-1/p$.
\end{thm}

\begin{proof}
Due to the estimate $\eqref{pri:3.2}$, our task is to estimate the layer and
co-layer type quantities in
the right-hand side. The easy one is
\begin{equation}\label{f:3.2}
\begin{aligned}
\|\nabla u_\varepsilon\|_{L^2(\Omega\setminus\Sigma_{4\varepsilon})}
&\leq C\varepsilon^{\frac{1}{2}-\frac{1}{p}}\Big(\int_\Omega|\nabla u_0|^p\Big)^{1/p}\\
&\leq Cr_0^{\frac{d}{2}-\frac{d}{p}}\Big(\frac{\varepsilon}{r_0}\Big)^{\frac{1}{2}-\frac{1}{p}}
\bigg\{r_0\|F\|_{L^p(\Omega)} + \|g\|_{W^{1-1/p,p}(\partial\Omega)} \bigg\},
\end{aligned}
\end{equation}
where we use the estimate $\eqref{pri:2.12}$ in the second step.

In view of the interior $H^2$ estimate (see for example
\cite[Chapter II, Theorem 1.1]{GM}), we have
\begin{equation}\label{f:3.3}
\dashint_{B(x,\delta(x)/4)}|\nabla^2 u_0|^2 dy
\leq \frac{C}{[\delta(x)]^{2}}\dashint_{B(x,\delta(x)/2)}|\nabla u_0|^2 dy
+ C\dashint_{B(x,\delta(x)/2)}|F|^2 dy
\end{equation}
for any $x\in\Sigma_{p_2\varepsilon}$, where $\delta(x)=\text{dist}(x,\partial\Omega)$.
Since $|y-x|\leq \delta(x)/4$, it is not hard to see that
$|\delta(x)-\delta(y)|\leq |x-y|\leq \delta(x)/4$ and
this gives that $(4/5)\delta(y)< \delta(x)< (4/3)\delta(y)$. Therefore,
\begin{equation*}
\begin{aligned}
\int_{\Sigma_{2\varepsilon}}|\nabla^2 u_0|^2 dx
\leq \int_{\Sigma_{2\varepsilon}}\dashint_{B(x,\delta(x)/4)}|\nabla^2 u_0|^2dy dx
\leq \int_{\Sigma_{\varepsilon}}|\nabla^2 u_0|^2 dx.
\end{aligned}
\end{equation*}
Then integrating both sides of $\eqref{f:3.3}$ over co-layer set $\Sigma_{p_2\varepsilon}$ leads to
\begin{equation*}
\begin{aligned}
\int_{\Sigma_{2\varepsilon}}|\nabla^2 u_0|^2 dx
&\leq C\int_{\Sigma_{\varepsilon}}|\nabla u_0|^2 [\delta(x)]^{-2} dx
+ C\int_\Omega |F|^2 dx \\
&\leq Cr_0^{d-1-\frac{2(d-1)}{p}}\varepsilon^{-1-\frac{2}{p}}\Big(\int_{\Omega}
|\nabla u_0|^pdx\Big)^{\frac{2}{p}}
+ \frac{C}{r_0^2}\int_{\Omega}|\nabla u_0|^2 dx
+ C\int_\Omega |F|^2 dx,
\end{aligned}
\end{equation*}
and this further gives
\begin{equation}\label{f:3.4}
\varepsilon\|\nabla^2 u_0\|_{L^2(\Sigma_{2\varepsilon})}
\leq Cr_0^{\frac{d}{2}-\frac{d}{p}}\Big(\frac{\varepsilon}{r_0}\Big)^{\frac{1}{2}-\frac{1}{p}}
\bigg\{r_0\|F\|_{L^p(\Omega)}+\|g\|_{W^{1-1/p,p}(\partial\Omega)}\bigg\},
\end{equation}
where we use the estimates $\eqref{pri:2.12}$ and $\eqref{f:2.11}$, as well as H\"older's inequality.

Combining the estimates $\eqref{pri:3.2}$, $\eqref{f:3.3}$ and $\eqref{f:3.4}$ gives the
stated estimate $\eqref{pri:3.3}$, by setting $\sigma = 1/2-1/p$. We have completed the proof.
\end{proof}

\noindent\textbf{Proof of Theorem $\ref{thm:1.1}$}.
If replacing $B(0,R)$ in Theorem $\ref{thm:3.1}$ by
a bounded $C^{1,1}$ domain,  then we can derive that
\begin{equation*}
\|u_\varepsilon - u_0\|_{L^2(\Omega)}
\leq C_\Omega\varepsilon^{1/2}\Big\{\|F\|_{L^2(\Omega)}
+\|g\|_{H^{3/2}(\partial\Omega)}\Big\},
\end{equation*}
where $C_\Omega$ depends on $\mu_0,\mu_2, d, r_0$ and the character of $\Omega$.
This in fact proved the estimate $\eqref{pri:1.2}$, while the proof of Theorem
$\ref{thm:3.2}$ gives the estimate $\eqref{pri:1.3}$. We have completed the proof.
\qed

\section{Interior estimates}
\begin{lemma}[approximating lemma I]\label{lemma:4.1}
Let $\sqrt[3]{\varepsilon}\leq r<(1/2)$.
Assume the same conditions as in Theorem $\ref{thm:4.1}$.
Let $u_\varepsilon\in H^1(B(0,2r))$ be a weak solution of
$\mathcal{L}_\varepsilon u_\varepsilon  = F$ in $B(0,2r)$.  Then there exists
$w\in H^1(B(0,r))$ such that
$\mathcal{L}_0 w  = F$,
and there holds
\begin{equation}\label{pri:4.1}
\begin{aligned}
 \Big(\dashint_{B(0,r)} |u_\varepsilon - w|^2  \Big)^{1/2}
 \leq C\left(\frac{\varepsilon}{r}\right)^{1/4}
 \bigg\{
 &\Big(\dashint_{B(0,2r)}|u_\varepsilon|^2 \Big)^{1/2}
 + r^2\Big(\dashint_{B(0,2r)} |F|^2\Big)^{1/2}\bigg\},
\end{aligned}
\end{equation}
where $C$ depends on $\mu_0,\mu_2$ and $d$.
\end{lemma}

\begin{proof}
The main idea may be found in \cite[Lemma 11.2]{SZ1}.
However, this result can not be obtained by rescaling arguments due to the nonlinearity of
$\mathcal{L}_\varepsilon$. On account of Caccioppoli's inequality $\eqref{pri:2.14}$ and
co-area formula, it is true that there exists $r_0\in[r,3r/2]$ such that
\begin{equation}\label{f:4.2}
\int_{\partial B(0,r_0)} |\nabla u_\varepsilon|^2 dS
\leq Cr^{-2}\int_{B(0,2r)}|u_\varepsilon|^2 dx +
Cr^{2}\int_{B(0,2r)} |F|^2 dx.
\end{equation}
Then for some $0<\delta\leq r$, we consider $\mathcal{L}_\varepsilon v_\varepsilon = F$ in $B(0,r_0)$ with
$v_\varepsilon = (u_\varepsilon)_\delta$ on $\partial B(0,r_0)$, and let
$\mathcal{L}_0 w = F$ in $B(0,r_0)$ with $w = (u_\varepsilon)_\delta$
on $\partial B(0,r_0)$,
where $(u_\varepsilon)_\delta\in H^{3/2}(\partial B(0,r_0))$ satisfies
the estimate $\eqref{pri:2.10}$. Consider
\begin{equation}\label{f:4.1}
\begin{aligned}
\int_{B(0,r)}|u_\varepsilon - w|^2 dx
&\leq \int_{B(0,r)}|u_\varepsilon - v_\varepsilon-z_\varepsilon|^2 dx
+ \int_{B(0,r)}|v_\varepsilon - w|^2 dx
+ \int_{B(0,r)}|z_\varepsilon|^2 dx \\
& =: I_1 + I_2 + I_3,
\end{aligned}
\end{equation}
where $z_\varepsilon\in H^1(\Omega)$ satisfies
\begin{equation*}
 \Delta z_\varepsilon = 0
 \quad\text{in} ~B(0,r_0),
 \qquad z_\varepsilon = u_\varepsilon-(u_\varepsilon)_\delta
 \quad \text{on}~\partial B(0,r_0).
\end{equation*}

We first handle $I_2$, and it follows from
the estimates $\eqref{pri:3.3}$ and $\eqref{pri:2.10}$ that
\begin{equation}\label{f:4.4}
\begin{aligned}
\sqrt{I_2} &\leq \|v_\varepsilon - w\|_{L^2(B(0,r_0))} \\
&\leq C\Big(\frac{\varepsilon}{r}\Big)^{1/2}
\bigg\{r^2\|F\|_{L^2(B(0,r_0))}
+ r\|(u_\varepsilon)_\delta\|_{H^{3/2}(\partial B(0,r_0))}\bigg\} \\
&\leq C\Big(\frac{\varepsilon}{\delta r}\Big)^{1/2}
\bigg\{r^2\|F\|_{L^2(B(0,2r))}
+ r\|u_\varepsilon\|_{H^{1}(\partial B(0,r_0))}\bigg\} \\
&\leq C\Big(\frac{\varepsilon}{\delta r}\Big)^{1/2}
\bigg\{\|u_\varepsilon\|_{L^{2}(B(0,2r))}
+ r^2\|F\|_{L^2(B(0,2r))}
\bigg\},
\end{aligned}
\end{equation}
in which the last step follows from the estimate $\eqref{f:4.2}$.

Before estimating $I_1$, we claim that
\begin{equation}\label{f:4.3}
\|\nabla u_\varepsilon - \nabla v_\varepsilon\|_{L^2(B(0,r_0))}
\leq C\|\nabla z_\varepsilon\|_{L^2(B(0,r_0))}
\end{equation}
where $C$ depends only $\mu_0$ and $\mu_2$. In fact,
\begin{equation*}
\int_{B(0,r_0)} \big[A(x/\varepsilon,\nabla u_\varepsilon)
-A(x/\varepsilon,\nabla v_\varepsilon)\big]\cdot\nabla \phi dx = 0
\end{equation*}
for any $\phi\in H_0^1(\Omega)$.
Set $\phi = u_\varepsilon - v_\varepsilon - z_\varepsilon$, and then
by applying the assumptions $\eqref{a:2}$ and $\eqref{a:3}$
to the above equation, we can arrive at the claim $\eqref{f:4.3}$ immediately.
Hence, from Poincar\'e's inequality and $\eqref{f:4.3}$, it follows that
\begin{equation}\label{f:4.5}
\begin{aligned}
\sqrt{I_1}
&\leq Cr\|\nabla(u_\varepsilon - v_\varepsilon - z_\varepsilon)\|_{L^2(B(0,r_0))}
\leq Cr\|\nabla z_\varepsilon\|_{L^2(B(0,r_0))} \\
&\leq Cr\|u_\varepsilon - (u_\varepsilon)_\delta\|_{H^{1/2}(\partial B(0,r_0))}
\leq Cr\delta^{1/2}\|u_\varepsilon\|_{H^1(\partial B(0,r_0))}\\
&\leq C\delta^{1/2}
\bigg\{\|u_\varepsilon\|_{L^2(B(0,2r))}+r^2\|F\|_{L^2(B(0,2r))}\bigg\},
\end{aligned}
\end{equation}
where we also use the estimates $\eqref{pri:2.10}$ and $\eqref{f:4.2}$, and
the second line is due to the fact $\|z_\varepsilon\|_{H^1(B(0,r_0))}
\leq C\|z_\varepsilon\|_{H^{1/2}(\partial B(0,r_0))}$.

The computation for $I_3$ relies on some properties of harmonic functions, and
\begin{equation}\label{f:4.6}
\begin{aligned}
\sqrt{I_3}
&\leq Cr^{1/2}\|z_\varepsilon\|_{L^{\frac{2d}{d-1}}(B(0,r_0))}
\leq Cr^{1/2}\|(z_\varepsilon)^*\|_{L^2(\partial B(0,r_0))}\\
&\leq Cr^{1/2}\|z_\varepsilon\|_{L^2(\partial B(0,r_0))}
\leq Cr^{1/2}\delta\|u_\varepsilon\|_{H^1(\partial B(0,r_0))}\\
&\leq Cr\delta^{1/2}\|u_\varepsilon\|_{H^1(\partial B(0,r_0))}
\leq C\delta^{1/2}
\bigg\{\|u_\varepsilon\|_{L^2(B(0,2r))}+r^2\|F\|_{L^2(B(0,2r))}\bigg\},
\end{aligned}
\end{equation}
in which the notation $(z_\varepsilon)^*$ represents the nontangential maximal
function of $z_\varepsilon$ (see for example \cite[Definition 2.19]{X3}).
Here the second inequality follows from \cite[Remark 9.3]{KFS1}, and
the third one is the so-called nontangential maximal function estimate
(see for example \cite[Theorem 7.5.14]{S4}). We employ the estimate
$\eqref{pri:2.13}$ in the fourth inequality and
the estimate $\eqref{f:4.2}$ in the last step.

Consequently, plugging the estimates $\eqref{f:4.4}$, $\eqref{f:4.5}$ and
$\eqref{f:4.6}$ back into $\eqref{f:4.1}$, we have
\begin{equation*}
\begin{aligned}
\|u_\varepsilon - w\|_{L^2(B(0,r))}
&\leq C\Big\{\delta^{1/2}+ \delta^{-1/2}
\big(\varepsilon/r\big)^{1/2}\Big\}
\bigg\{\|u_\varepsilon\|_{L^2(B(0,r_0))}+r^2\|F\|_{L^2(B(0,r_0))}\bigg\}\\
&\leq C(\varepsilon/r)^{1/4}
\bigg\{\|u_\varepsilon\|_{L^2(B(0,2r))}+r^2\|F\|_{L^2(B(0,2r))}\bigg\},
\end{aligned}
\end{equation*}
in which the second line asks for $\delta = (\varepsilon/r)^{1/2}$, and
the assumption $\sqrt[3]{\varepsilon}\leq r<(1/2)$ meets this requirement.
By multiplying $r^{-d/2}$ in both sides of the above inequality,
the desired estimate $\eqref{pri:4.1}$ follows, and we have completed the proof.
\end{proof}

Before we proceed further, for any matrix $M\in \mathbb{R}^{d}$,
we denote $G(r,v)$ as the following
\begin{equation}
\begin{aligned}
G(r,v) &= \frac{1}{r}\inf_{M\in\mathbb{R}^{d}\atop c\in\mathbb{R}}
\Bigg\{\Big(\dashint_{B(0,r)}|v-Mx-c|^2dx\Big)^{\frac{1}{2}}
+ r^2\Big(\dashint_{B(0,r)}|F|^p\Big)^{\frac{1}{p}}\Bigg\}.
\end{aligned}
\end{equation}

\begin{lemma}\label{lemma:4.2}
Given $F\in L^p(\Omega)$ for some $p>d$,
let $u_0\in H^1(B(0,2r))$ be a solution of
$\mathcal{L}_0 u_0 = F$ in $B(0,2r)$.
Then
there exists $\theta\in(0,1/4)$,
depending on $\mu_0,\mu_1, p, d$, such that
\begin{equation}\label{pri:4.3}
 G(\theta r,u_0)\leq \frac{1}{2} G(r,u_0)
\end{equation}
holds for any $r\in(0,1)$.
\end{lemma}

\begin{proof}
It is fine to assume $u_0\in H^2(B(0,r))$ and we have the following equation
\begin{equation}\label{pde:4.1}
\int_{B(0,r)}\nabla_{\xi_j}\widehat{A}^i(\nabla u_0)\nabla^2_{jk}u_0\nabla_i\phi dx =
-\int_{B(0,r)} F \nabla_k\phi dx
\end{equation}
for any $\phi\in H_0^1(B(0,r))$, and $k=1,\cdots,d$.
Let $\tilde{a}_{ij}(x) = \nabla_{\xi_j}\widehat{A}^i(\nabla u_0)$, which
will give a linear operator with the uniform ellipticity on account of
$\eqref{pri:2.3}$ and $\eqref{a:4}$. Hence, the De Giorgi-Nash-Moser theorem
tells us that for any $p>d$, there exists $\alpha\in(0,1)$ and $C>1$, depending only on
$\mu_0,\mu_2, d$ and $p$, such that
\begin{equation}\label{pri:4.2}
 [\nabla u_0]_{C^{0,\alpha}(B(0,r/2))}
 \leq Cr^{-\alpha}\bigg\{\frac{1}{r}\Big(\dashint_{B(0,r)}|u_0|^2\Big)^{1/2}
 + r\Big(\dashint_{B(0,r)}|F|^p\Big)^{1/p}\bigg\}
\end{equation}
(see for example \cite[Theorem 8.13]{MGLM}).
By the definition of $G(\theta r,u_0)$, we see that
\begin{equation}\label{f:4.7}
\begin{aligned}
G(\theta r,u_0) &\leq  \frac{1}{\theta r}
\bigg\{\Big(\dashint_{B(0,\theta r)}|u_0-M_0x-c_0|^2\Big)^{\frac{1}{2}}
+ \theta^2r^2\Big(\dashint_{B(0,\theta r)}|F|^p\Big)^{\frac{1}{p}}\bigg\} \\
&\leq \theta^{\sigma}\bigg\{r^\alpha\big[u_0\big]_{C^{1,\alpha}(B(0,r/4))}
+r\Big(\dashint_{B(0,r)}|F|^pdx\Big)^{\frac{1}{p}}
\bigg\},
\end{aligned}
\end{equation}
where $\sigma = \min\{1-d/p,\alpha\}$.
Here we choose $M_0 = \nabla u_0(0)$ and $c_0 = u_0(0)$.
For any
$M\in \mathbb{R}^{d}$ and $c\in\mathbb{R}$, let $\tilde{u}_0 = u_0- Mx-c$.
Obviously, it satisfies the equation $\eqref{pde:4.1}$ and by $\eqref{pri:4.2}$,
\begin{equation*}
\begin{aligned}
r^\alpha\big[u_0\big]_{C^{1,\alpha}(B(0,r/2))}
= r^\alpha\big[\tilde{u}_0\big]_{C^{1,\alpha}(B(0,r/2))}
\leq CG(r,u_0).
\end{aligned}
\end{equation*}
This together with $\eqref{f:4.7}$ leads to
\begin{equation*}
G(\theta r,u_0) \leq C\theta^\sigma  G(r,u_0),
\end{equation*}
and we are done.
\end{proof}

For simplicity, we also denote $\Phi(r)$ by
\begin{equation*}
\begin{aligned}
\Phi(r) = \frac{1}{r}\inf_{q\in\mathbb{R}}\bigg\{
\Big(\dashint_{D_{r}}|u_\varepsilon - c|^2 \Big)^{1/2}
+ r^2\Big(\dashint_{D_{r}} |F|^p\Big)^{1/p}\bigg\}.
\end{aligned}
\end{equation*}

\begin{lemma}\label{lemma:4.4}
Assume the same conditions as in Theorem $\ref{thm:4.1}$.
Let $u_\varepsilon$ be the solution of
$\mathcal{L}_\varepsilon(u_\varepsilon) = F$ in $B(0,2r)$.
Then we have
\begin{equation}
 G(\theta r, u_\varepsilon) \leq \frac{1}{2}G(r,u_\varepsilon)
 + C\left(\frac{\varepsilon}{r}\right)^{1/4}\Phi(2r)
\end{equation}
for any $\sqrt[3]{\varepsilon}\leq r<(1/4)$, where
$0<\theta<(1/4)$ has been given in Lemma $\ref{lemma:4.2}$.
\end{lemma}

\begin{proof}
Fix $r\in[\sqrt[3]{\varepsilon},1/4)$, let $w$ be a solution to
$\mathcal{L}_0 w = F$ in $B(0,r)$ as in Lemma $\ref{lemma:4.1}$.
Then we obtain
\begin{equation*}
\begin{aligned}
G(\theta r,u_\varepsilon)
&\leq \frac{1}{\theta r}\Big(\dashint_{B(0,\theta r)}
|u_\varepsilon - w|^2 \Big)^{\frac{1}{2}}
+ G(\theta r, w) \\
&\leq \frac{C}{r}\Big(\dashint_{B(0,r)}|u_\varepsilon - w|^2\Big)^{\frac{1}{2}}
+ \frac{1}{2}G(r, w)\\
&\leq \frac{1}{2}G(r, u_\varepsilon)
+ \frac{C}{r}\Big(\dashint_{B(0,r)}|u_\varepsilon - w|^2\Big)^{\frac{1}{2}}\\
&\leq  \frac{1}{2}G(r, u_\varepsilon) + C\Big(\frac{\varepsilon}{r}\Big)^{1/4}
 \bigg\{
 \frac{1}{r}\Big(\dashint_{B(0,2r)}|u_\varepsilon|^2 \Big)^{1/2}
 + r\Big(\dashint_{B(0,2r)} |F|^p\Big)^{1/p}\bigg\},
\end{aligned}
\end{equation*}
where we use the estimate $\eqref{pri:4.3}$ in the second inequality,
and $\eqref{pri:4.1}$ in the last one. Note that for any $c\in\mathbb{R}$,
$u_\varepsilon - c$ is still a solution of
$\mathcal{L}_\varepsilon u_\varepsilon = F$ in $B(0,2r)$, and the proof is complete.
\end{proof}

\begin{lemma}[iteration lemma]\label{lemma:4.3}
Let $\Psi(r)$ and $\psi(r)$ be two nonnegative continuous functions on the integral $(0,1]$.
Let $0<\varepsilon<\frac{1}{4}$. Suppose that there exists a constant $C_0$ such that
\begin{equation}\label{pri:4.4}
\left\{\begin{aligned}
  &\max_{r\leq t\leq 2r} \Psi(t) \leq C_0 \Psi(2r),\\
  &\max_{r\leq s,t\leq 2r} |\psi(t)-\psi(s)|\leq C_0 \Psi(2r),
  \end{aligned}\right.
\end{equation}
We further assume that
\begin{equation}\label{pri:4.5}
\Psi(\theta r)\leq \frac{1}{2}\Psi(r) + C_0w(\varepsilon/r)
\Big\{\Psi(2r)+\psi(2r)\Big\}
\end{equation}
holds for any $\varepsilon\leq r <(1/4)$, where $\theta\in(0,1/4)$ and $w$ is a nonnegative
increasing function in $[0,1]$ such that $w(0)=0$ and
\begin{equation*}
 \int_0^1 \frac{w(t)}{t} dt <\infty.
\end{equation*}
Then, we have
\begin{equation}
\max_{\varepsilon\leq r\leq 1}\Big\{\Psi(r)+ \psi(r)\Big\}
\leq C\Big\{\Psi(1)+\psi(1)\Big\},
\end{equation}
where $C$ depends only on $C_0, \theta$ and $w$.
\end{lemma}

\begin{proof}
The proof may be found in \cite[Lemma 8.5]{S5}.
\end{proof}

\noindent\textbf{Proof of Theorem $\ref{thm:1.2}$}.
It is fine to assume $0<\varepsilon<1/4$, otherwise it follows from the classical theory.
In view of Lemma $\ref{lemma:4.3}$,
we set $\Psi(r) = G(r,u_\varepsilon)$, $w(t) =t^{1/4}$.
To prove the desired estimate
$\eqref{pri:1.1}$, it is sufficient to verify $\eqref{pri:4.4}$
and $\eqref{pri:4.5}$.
Let $\psi(r) = |M_r|$, where $M_r$ is the matrix associated with $\Psi(r)$, respectively.
\begin{equation*}
\begin{aligned}
\Psi(r) &= \frac{1}{r}
\inf_{c\in\mathbb{R}}\Bigg\{\Big(\dashint_{B(0,r)}|u_\varepsilon-M_r x - c|^2\Big)^{\frac{1}{2}}
+ r^2\Big(\dashint_{B(0,r)}|F|^p\Big)^{\frac{1}{p}}\Bigg\},
\end{aligned}
\end{equation*}
Then we have
\begin{equation*}
 \Phi(r) \leq C\Big\{\Psi(2r) + \psi(2r)\Big\},
\end{equation*}
This coupled with Lemma $\ref{lemma:4.4}$ leads to
\begin{equation*}
\Psi(\theta r)\leq \frac{1}{2}\Psi(r) + C_0 w(\varepsilon/r)\Big\{\Psi(2r)+\psi(2r)\Big\},
\end{equation*}
for $\sqrt[3]{\varepsilon}\leq r<(1/4)$,
which satisfies the condition $\eqref{pri:4.5}$ in Lemma $\ref{lemma:4.3}$.
Let $t,s\in [r,2r]$, and $v(x)=(M_t-M_s)x$. It is clear to see $v$ is harmonic in $\mathbb{R}^d$,
and we have
\begin{equation}\label{f:4.8}
\begin{aligned}
|M_t-M_s|&\leq \frac{C}{r}\Big(\dashint_{B(0,r)}|(M_t-M_s)x-c|^2\Big)^{\frac{1}{2}}\\
&\leq \frac{C}{t}
\Big(\dashint_{B(0,r)}|u_\varepsilon - M_tx-c|^2\Big)^{\frac{1}{2}}
+ \frac{C}{s}\Big(\dashint_{B(0,r)}|u_\varepsilon - M_sx-c|^2\Big)^{\frac{1}{2}}\\
&\leq C\Big\{\Psi(t)+\Psi(s)\Big\}\leq C\Psi(2r),
\end{aligned}
\end{equation}
where the second and the last steps are based on the fact that $s,t\in[r,2r]$. Due to the same reason, it
is easy to obtain $\Psi(r)\leq C\Psi(2r)$, where we use the assumption $p>d$.
The estimate $\eqref{f:4.8}$ satisfies the condition
$\eqref{pri:4.2}$.
Hence, according to Lemma $\ref{lemma:4.2}$,
for any $r\in[\sqrt[3]{\varepsilon},1)$,
we have the following estimate
\begin{equation}
\begin{aligned}
\frac{1}{r}\inf_{c\in\mathbb{R}}\Big(\dashint_{B(0,r)}|u_\varepsilon - c|^2 \Big)^{\frac{1}{2}}
&\leq \Big\{\Psi(r) + \psi(r)\Big\}
\leq C\Big\{\Psi(1) + \psi(1)\Big\}\\
&\leq C\bigg\{\Big(\dashint_{B(0,1)}|\nabla u_\varepsilon|^2\Big)^{1/2}
+ \|F\|_{L^p(B(0,1))}\bigg\}
\end{aligned}
\end{equation}
Hence, the desired estimate $\eqref{pri:1.1}$ consequently follows from
the above estimate and Caccioppoli's inequality $\eqref{pri:2.14}$,
and we have completed the proof.
\qed

\begin{thm}[interior $W^{1,p}$ estimates]\label{thm:4.1}
Let $B(0,2)\subset\Omega$ and $2\leq p<\infty$.
Suppose that $\mathcal{L}_\varepsilon$ satisfies
$\eqref{a:1}$, $\eqref{a:2}$, and $\eqref{a:3}$.
Let $u_\varepsilon\in H^1(B(0,2))$ be a weak solution of
$\mathcal{L}_\varepsilon u_\varepsilon = \emph{div}(f)+F$ in $B(0,2)$,
where $f\in L^p(B(0,2);\mathbb{R}^d)$ and $F\in L^p(B(0,2))$.
Then $|\nabla u_\varepsilon|\in L^p_{loc}(B(0,2))$, and there holds
\begin{equation}\label{pri:4.6}
\begin{aligned}
\Big(\dashint_{B(0,r)} |\nabla u_\varepsilon|^p\Big)^{\frac{1}{p}}
\leq C_p\bigg\{\Big(\dashint_{B(0,2r)} |\nabla u_\varepsilon|^2\Big)^{\frac{1}{2}}
+ \Big(\dashint_{B(0,2r)} |f|^p\Big)^{\frac{1}{p}}
+r\Big(\dashint_{B(0,2r)} |F|^p\Big)^{\frac{1}{p}}\bigg\}
\end{aligned}
\end{equation}
for any $0<r<(1/2)$,
where $C_p$ depends only on $\mu_0,\mu_1,\mu_2,\tau,d$ and $p$.
\end{thm}

\begin{proof}
Step 1. Consider the estimate $\eqref{pri:4.6}$ in the case of $f=0$ and $F=0$.
Since the assumption $\eqref{a:4}$ satisfies $(\delta,R)$-vanishing condition of
$A(\cdot,\xi)$, it follows from \cite[Lemma 2.5]{BW} and the Lipschitz estimate
$\eqref{pri:1.1}$ that
\begin{equation*}
\Big(\dashint_{B(0,\sqrt[3]{\varepsilon})}|\nabla u_\varepsilon|^p\Big)^{1/p}
\leq C_p\Big(\dashint_{B(0,2\sqrt[3]{\varepsilon})}|\nabla u_\varepsilon|^2\Big)^{1/2}
\leq C_p\Big(\dashint_{B(0,R)}|\nabla u_\varepsilon|^2\Big)^{1/2},
\end{equation*}
where $2\sqrt[3]{\varepsilon}\leq R<1$.
Then by a simple covering argument we may have
\begin{equation}\label{f:4.10}
\Big(\dashint_{B(0,r)}|\nabla u_\varepsilon|^p\Big)^{1/p}
\leq C_p\Big(\dashint_{B(0,2r)}|\nabla u_\varepsilon|^2\Big)^{1/2}
\end{equation}
for any $0<r<(1/2)$, where $C_p$ depends on $\mu_0,\mu_1,\mu_2,\tau,d$ and $p$.

Step 2. We plan to
apply a real method (see \cite[Theorem 3.2.3]{S4}) to handle the case of $f\not = 0$.
To do so, construct a solution $v_\varepsilon\in H^1(B(0,2r))$ such that
\begin{equation*}
\left\{\begin{aligned}
\mathcal{L}_\varepsilon v_\varepsilon &= 0 &\quad&\text{in}~~B(0,2r),\\
v_\varepsilon &= u_\varepsilon  &\quad&\text{on}~\partial B(0,2r).
\end{aligned}\right.
\end{equation*}
Thus by definition there holds
\begin{equation*}
\int_{B(0,2r)}\big[A(x/\varepsilon,\nabla u_\varepsilon)-
A(x/\varepsilon,\nabla v_\varepsilon)\big]\nabla\varphi dz
= -\int_{B(0,2r)}f\cdot\nabla\varphi dz + \int_{B(0,2r)}F\varphi dz
\end{equation*}
for any $\varphi\in H_0^1(B(0,2r))$. Set $\varphi = u_\varepsilon - v_\varepsilon$,
and then in view of $H^1$ theory (see Theorem $\ref{thm:2.1}$) one may derive
\begin{equation}\label{f:4.9}
\int_{B(0,2r)} |\nabla u_\varepsilon - \nabla v_\varepsilon|^2 dz
\leq C(\mu_0,d)\bigg\{\int_{B(0,2r)} |f|^2 dz
+r^2\int_{B(0,2r)} |F|^2 dz\bigg\},
\end{equation}
where we use the assumption $\eqref{a:2}$ and Poincar\'e's inequality
coupled with Young's inequality. As in $\eqref{def:2.4}$, let
\begin{equation*}
  \tilde{F}(y) = |F(y)|\text{dist}(y,M_{3r}^{4r}(0))
+ |f(y)|,
\end{equation*}
where $M_{3r}^{4r}(0) = B(0,4r)\setminus B(0,3r)$.
The estimate $\eqref{f:4.9}$ together with the estimate $\eqref{f:4.10}$ gives
\begin{equation}\label{f:4.11}
\begin{aligned}
\Big(\dashint_{B(0,r)}|\nabla v_\varepsilon|^p\Big)^{1/p}
&\leq C_p\Big(\dashint_{B(0,2r)}|\nabla v_\varepsilon|^2\Big)^{1/2} \\
&\leq C_p\bigg\{\dashint_{B(0,2r)}|\nabla u_\varepsilon|^2\Big)^{1/2}
+ \Big(\dashint_{B(0,2r)} |\tilde{F}|^2\Big)^{1/2} \bigg\}.
\end{aligned}
\end{equation}

Consequently, the stated estimate $\eqref{pri:4.6}$ follows from the estimates $\eqref{f:4.9}$,
$\eqref{f:4.11}$, and \cite[Theorem 3.2.3]{S4},
and we may end the proof here.
\end{proof}

\section{Boundary estimates}

\begin{lemma}\label{lemma:5.3}
Let $\Omega$ be a bounded Lipschitz domain. Suppose that $A$ satisfies
$\eqref{a:1}$ and $\eqref{a:2}$. Let $u_\varepsilon\in H^1(B(x,1)\cap\Omega)$
be a weak solution of $\mathcal{L}_\varepsilon u_\varepsilon = \emph{div}(f)+F$
in $B(x,1)\cap\Omega$ with
$u_\varepsilon = 0$ on $B(x,1)\cap\partial\Omega$, where $x\in\partial\Omega$,
$f\in L^{2}(\Omega;\mathbb{R}^d)$ and $F\in L^2(\Omega)$.
Then for any $0<r\leq (1/2)$,
we have
\begin{equation}\label{pri:5.1}
\Big(\dashint_{B(x,r)\cap \Omega}|\nabla u_\varepsilon|^2 \Big)^{\frac{1}{2}}
\leq C
\bigg\{ \frac{1}{r}\Big(\dashint_{B(x,2r)\cap\Omega}
|u_\varepsilon|^2\Big)^{\frac{1}{2}}
+\Big(\dashint_{B(x,2r)\cap\Omega}|f|^2\Big)^{\frac{1}{2}}
+ r\Big(\dashint_{B(x,2r)\cap\Omega}|F|^2\Big)^{\frac{1}{2}}\bigg\},
\end{equation}
where $C_\alpha$ depends on $\mu_0,\mu_2,d,\alpha$ and the character of $\Omega$.
\end{lemma}

\begin{proof}
The proof of this result is quite similar to that given earlier for
the interior one (see Lemma $\ref{lemma:2.3}$), and so is omitted.
\end{proof}

\begin{lemma}[approximating lemma II]\label{lemma:5.1}
Let $\Omega$ be a bounded Lipschitz domain and $\varepsilon\leq r\leq(1/2)$.
Suppose that $\mathcal{L}_\varepsilon$ satisfies
the conditions $\eqref{a:1}$ and $\eqref{a:2}$.
Let $u_\varepsilon\in H^1(B(0,2r)\cap\Omega)$ be a weak solution of
$\mathcal{L}_\varepsilon u_\varepsilon  = F$ in $B(0,2r)\cap\Omega$ with
$u_\varepsilon = 0$ on $B(0,2r)\cap\partial\Omega$. Then there exists
$w\in H^1(B(0,r)\cap\Omega)$ and some $p>2$ such that
there holds
\begin{equation}\label{pri:5.2}
\begin{aligned}
 \Big(\dashint_{B(0,r)\cap\Omega} |u_\varepsilon - w|^2  \Big)^{1/2}
 \leq C\left(\frac{\varepsilon}{r}\right)^{\sigma}
 \bigg\{
 &\Big(\dashint_{B(0,2r)\cap\Omega}|u_\varepsilon|^2 \Big)^{1/2}
 + r^2\Big(\dashint_{B(0,2r)\cap\Omega} |F|^p\Big)^{1/p}\bigg\},
\end{aligned}
\end{equation}
where $\sigma=1/2-1/p$, and $C$ depends on $\mu_0,\mu_2, d$ and the character of
$\Omega$.
\end{lemma}

\begin{proof}
Let $D_r = B(0,r)\cap\Omega$ and $\Delta_r = B(0,r)\cap\partial\Omega$.
In view of the co-area formula and Caccioppoli's inequality
$\eqref{pri:5.1}$, there exists $r_0\in[3r/2,2r]$
such that
\begin{equation}\label{f:5.1}
\int_{\partial D_{r_0}\setminus\Delta_{r_0}}|\nabla u_\varepsilon|^2 dS
\leq \frac{C}{r^2}\int_{D_{2r}}|u_\varepsilon|^2 dx
+ Cr^2 \int_{D_{2r}}|F|^2 dx.
\end{equation}
Then we construct $w\in H^1(D_{r_0})$ such that $\mathcal{L}_0 w = F$ in $D_{r_0}$ and
$w = u_\varepsilon$ on $\partial D_{r_0}$. Note that we choose $p>2$ such that
$\|u_\varepsilon\|_{W^{1-1/p,p}(\partial D_{r_0})}\leq C\|u_\varepsilon\|_{H^1(\partial D_{r_0})}$, and we have
\begin{equation*}
\begin{aligned}
\|u_\varepsilon - w\|_{L^2(D_r)}
&\leq \|u_\varepsilon - w\|_{L^2(D_{r_0})}\\
&\leq Cr^{d\sigma}\Big(\frac{\varepsilon}{r}\Big)^{\sigma}
\bigg\{r^2\|F\|_{L^p(D_{2r})}
+r\|\nabla u_\varepsilon\|_{H^1(\partial D_{r_0}\setminus\Delta_{r_0})}\bigg\}\\
&\leq  Cr^{d\sigma}\Big(\frac{\varepsilon}{r}\Big)^{\sigma}
\bigg\{\|u_\varepsilon\|_{L^2(D_{2r})} + r^2\|F\|_{L^p(D_{2r})}\bigg\}
\end{aligned}
\end{equation*}
where $\sigma = 1/2-1/p$, and we use the estimate $\eqref{f:5.1}$ in the last step. This implies the stated estimate
$\eqref{pri:5.2}$, and we have completed the proof.
\end{proof}

\begin{lemma}\label{lemma:5.2}
Let $\Omega$ be a bounded $C^1$ domain and $0<r<(r_0/4)$.
Suppose that $A$ satisfies
$\eqref{a:1}$ and $\eqref{a:2}$.
Let $w$ be a weak solution of
$\mathcal{L}_0 w = F$ in $\Omega\cap B(0,2r)$ and $w = 0$ on $B(0,2r)\cap\partial\Omega$,
where $F\in L^\infty(\Omega)$.
Then for any $\beta\in(0,1)$ we have
\begin{equation}\label{pri:5.3}
 \big[w\big]_{C^{0,\beta}(B(0,r)\cap\Omega)}
 \leq C_{\beta}r^{-\beta}
 \bigg\{\Big(\dashint_{B(0,2r)\cap\Omega}|w|^2\Big)^{1/2}
 +r^2\big\|F\big\|_{L^\infty(B(0,2r\cap\Omega))}\bigg\},
\end{equation}
where $C_\beta$ depends on $\mu_0,\mu_2,d,\beta$ and the character of $\Omega$.
\end{lemma}

\begin{proof}
In the case of $F=0$, by \cite[Theorem 3.8]{BW}, for any $2\leq p<\infty$
it is not hard to derive
\begin{equation*}
\Big(\dashint_{B(0,r)\cap\Omega}|\nabla w|^p\Big)^{1/p}
\leq C_p\Big(\dashint_{B(0,2r)\cap\Omega}|\nabla w|^2\Big)^{1/2}.
\end{equation*}
Then the case of $F\not=0$ will follows from a real method as we did
for Theorems $\ref{thm:4.1}$ and $\ref{thm:5.3}$, and therefore no proof
will be given for the following estimate
\begin{equation}\label{f:5.6}
\Big(\dashint_{B(0,r)\cap\Omega}|\nabla w|^p\Big)^{1/p}
\leq C_p\bigg\{\Big(\dashint_{B(0,2r)\cap\Omega}|\nabla w|^2\Big)^{1/2}
+r\Big(\dashint_{B(0,2r)\cap\Omega}|F|^p\Big)^{1/p}\bigg\}.
\end{equation}
Hence, it follows from the Sobolev theorem that
\begin{equation*}
\begin{aligned}
\big[w\big]_{C^{0,\beta}(B(0,r)\cap\Omega)}
&\leq C\|\nabla w\|_{L^p(B(0,r)\cap\Omega)}
= Cr^{\frac{d}{p}}\Big(\dashint_{B(0,r)\cap\Omega} |\nabla w|^p\Big)^{1/p}\\
&\leq Cr^{\frac{d}{p}}
\bigg\{\Big(\dashint_{B(0,3r/2)\cap\Omega} |\nabla w|^2\Big)^{1/2}
+ r\Big(\dashint_{B(0,3r/2)\cap\Omega}|F|^p\Big)^{1/p}\bigg\}\\
&\leq Cr^{-\beta}\bigg\{\Big(\dashint_{B(0,2r)\cap\Omega} |w|^2\Big)^{1/2}
+r^2\big\|F\big\|_{L^\infty(B(0,2r\cap\Omega))}\bigg\}
\end{aligned}
\end{equation*}
where $\beta = 1-d/p$ for any $p>d$, and we use the $W^{1,p}$ estimate
$\eqref{f:5.6}$ in the second inequality, and Caccioppoli's inequality $\eqref{pri:5.1}$
in the last step.
The proof is complete.
\end{proof}

\begin{thm}
Let $\Omega$ be a bounded $C^1$ domain. Suppose that
$\mathcal{L}_\varepsilon$ satisfies the conditions $\eqref{a:1}$ and $\eqref{a:2}$.
Let $u_\varepsilon$ be a weak solution of
$L_\varepsilon u_\varepsilon = F$ in $\Omega\cap B(0,2)$
and $u_\varepsilon = 0$ on $B(0,2)\cap\partial\Omega$,
where $F\in L^\infty(\Omega)$. Then for any $\alpha\in(0,1)$
and $\varepsilon\leq r< R\leq(1/2)$, we have
\begin{equation}\label{pri:5.4}
 \Big(\dashint_{B(0,r)\cap\Omega} |\nabla u_\varepsilon|^2\Big)^{1/2}
 \leq C_{\alpha}\Big(\frac{r}{R}\Big)^{\alpha-1}
 \bigg\{
 \Big(\dashint_{B(0,R)\cap\Omega}|\nabla u_\varepsilon|^2\Big)^{1/2}
 + R\big\|F\big\|_{L^\infty(B(0,R)\cap\Omega)}\bigg\}.
\end{equation}
Moreover, if we additionally assume that $\mathcal{L}_\varepsilon$ satisfies
the smoothness condition $\eqref{a:3}$ and $F=0$,
then there holds
\begin{equation}\label{pri:5.5}
\big[u_\varepsilon\big]_{C^{0,\alpha}(B(0,r/2)\cap\Omega)}
\leq C_\alpha\Big(\dashint_{B(0,r)\cap\Omega}|u_\varepsilon|^2\Big)^{1/2}
\end{equation}
for any $0<r<(1/2)$,
where $C_\beta$ depends on $\mu_0,\mu_2,d,\alpha$ and the character of $\Omega$.
\end{thm}

\begin{proof}
The main idea may be found in \cite[Theorem 5.2]{S5}, which actually could be extended to
the nonhomogeneous cases, and we provide a proof for the sake of the completeness.
To do so, we set $\alpha\in (0,1)$, and
\begin{equation*}
H(r,\sigma,v) = r^{-\alpha}\bigg\{\Big(\dashint_{B(0,r)\cap\Omega}|v|^2 dx\Big)^{\frac{1}{2}}
+ r^2\|F\|_{L^\infty(B(0,r)\cap\Omega)}
\bigg\}.
\end{equation*}
For each $r\in[\varepsilon,1/2]$,
let $u_0 = w$ be the function given in Lemma $\ref{lemma:5.1}$. Then
for any $\alpha<\beta<1$,
it follows from the boundary H\"older estimate $\eqref{pri:5.3}$ that
\begin{equation*}
\big[w\big]_{C^{0,\beta}(D_{r/2})}
\leq CH(r,\beta,w).
\end{equation*}
For any $\theta\in(0,1/4)$, the above estimate gives
\begin{equation}\label{f:5.3}
H(\theta r,\alpha,w)
\leq C\theta^{\beta-\alpha} H(r,\alpha,w).
\end{equation}

Then we find that
\begin{equation}
\begin{aligned}
H(\theta r,\alpha,u_\varepsilon)
&\leq (\theta r)^{-\alpha}\Big(\dashint_{B(0,\theta r)\cap\Omega}
|u_\varepsilon - w|^2 dx\Big)^{\frac{1}{2}}
+ H(\theta r,\alpha,w) \\
&\leq C\theta^{-\alpha-\frac{d}{2}}(\varepsilon/r)^{\sigma}
H(2r,\alpha,u_\varepsilon)
+ C\theta^{\beta-\alpha} H(r,\alpha,w)\\
&\leq \max\big\{C\theta^{-\alpha-\frac{d}{2}}(\varepsilon/r)^{\sigma}
+C_1\theta^{\beta-\alpha}(\varepsilon/r)^{\sigma}
+C_2\theta^{\beta-\alpha}\big\}H(2r,\alpha,u_\varepsilon),
\end{aligned}
\end{equation}
where $\sigma\in(0,1)$ is given in Lemma $\ref{lemma:5.1}$,
and we use the estimates $\eqref{pri:5.2}$ and $\eqref{f:5.3}$
in the second step.
Now, we first fix $\theta\in(0,1/4)$ such that $C_2\theta^{\beta-\alpha}\leq 1/4$, and
then let $r\geq N\varepsilon$. By choosing $N>1$ large, we also obtain
\begin{equation*}
\max\big\{C\theta^{-\alpha-\frac{d}{2}}(\varepsilon/r)^{\sigma},
~C_1\theta^{\beta-\alpha}(\varepsilon/r)^{\sigma}\big\}\leq 1/4.
\end{equation*}
Hence,
\begin{equation}
H(\theta r,\alpha,u_\varepsilon)\leq \frac{1}{4} H(2r,\alpha,u_\varepsilon).
\end{equation}
Moreover, multiplying $r^{-1}$ on the both sides of $\eqref{f:4.6}$ and then integrating
with respect to $r$ from $N\varepsilon$ to $R/2$, we have
\begin{equation*}
\int_{\theta N\varepsilon}^{\frac{\theta R}{2}}H(r,\alpha,u_\varepsilon)\frac{dr}{r}
\leq \frac{1}{4} \int_{2N\varepsilon}^{R}H(r,\alpha,u_\varepsilon)\frac{dr}{r}
\end{equation*}
and this implies that
\begin{equation*}
\int_{\theta N\varepsilon}^{R}H(s,\alpha,u_\varepsilon)\frac{ds}{s}
\leq \frac{4}{3} \int_{\frac{\theta R}{2}}^{R}H(s,\alpha,u_\varepsilon)\frac{ds}{s}
\leq CH(R,\sigma,u_\varepsilon).
\end{equation*}
Thus we deduce from the above estimate that
$H(r,\alpha,u_\varepsilon)\leq CH(R,\sigma,u_\varepsilon)$
for any $\varepsilon\leq r< R/2$.
Then it is clear to see that
\begin{equation*}
\Big(\dashint_{B(0,r)\cap\Omega}|\nabla u_\varepsilon|^2 dx\Big)^{\frac{1}{2}}
\leq Cr^{\alpha-1}H(2r,\alpha,u_\varepsilon)\leq
Cr^{\alpha-1}H_0(R,\alpha,u_\varepsilon)
\end{equation*}
where we use Caccioppoli's inequality $\eqref{pri:5.1}$ in the first step.
The desired estimate $\eqref{pri:5.4}$ consequently follows
from Poincar\'e's inequality.

In terms of $\eqref{pri:5.5}$, we just give a remark here. The assumption
$\eqref{a:3}$ in fact verified the so-called $(\delta,R)$-vanishing condition,
that means we may derive the related H\"older estimate
in small scales. As in the proof of Lemma $\ref{lemma:5.2}$,
it follows from \cite[Theorem 3.8]{BW} that
\begin{equation*}
 \big[u_\varepsilon\big]_{C^{0,\alpha}(B(0,\varepsilon)\cap\Omega)}
 \leq C\varepsilon^{1-\alpha}\Big(\dashint_{B(0,2\varepsilon)\cap\Omega}
 |\nabla u_\varepsilon|^2\Big)^{1/2},
\end{equation*}
and the remainder of the proof is as the same as in \cite[Corollary 5.2]{S5}.
We just end the proof here.
\end{proof}

\begin{lemma}
Let $2\leq p<\infty$. Suppose that $\mathcal{L}_\varepsilon$ satisfies
the assumptions $\eqref{a:1}$, $\eqref{a:2}$ and $\eqref{a:3}$. Let $u_\varepsilon$ be
a weak solution of $L_\varepsilon u_\varepsilon = 0$ in $B(0,1)\cap\Omega$ with
$u_\varepsilon = 0$ on $B(0,1)\cap\partial\Omega$.
Then there holds
\begin{equation}\label{pri:5.6}
\Big(\dashint_{B(0,r)\cap\Omega}|\nabla u_\varepsilon|^p\Big)^{1/p}
\leq C_p\Big(\dashint_{B(0,2r)\cap\Omega}|\nabla u_\varepsilon|^2\Big)^{1/2}
\end{equation}
for any $0<r<(1/2)$,
where $C_p$ depends on $\mu_0, \mu_1, \mu_2, d, p$ and the character of $\Omega$.
\end{lemma}

\begin{proof}
The main idea may be found in \cite[Lemma 5.4]{S5}, and we provide a proof for the sake of
the completeness. First of all, it follows from
the estimates $\eqref{pri:4.6}$ and $\eqref{pri:5.4}$ that
\begin{equation*}
\begin{aligned}
\Big(\dashint_{B(y,\delta(y)/8)}|\nabla u_\varepsilon|^p dx\Big)^{\frac{1}{p}}
&\leq C\Big(\dashint_{B(y,\delta(y)/4)}|\nabla u_\varepsilon|^2 dx\Big)^{\frac{1}{2}}\\
&\leq C\Big(\dashint_{D(y^\prime,2\delta(y))}|\nabla u_\varepsilon|^2 dx\Big)^{\frac{1}{2}}
 \leq C\Big(\frac{\delta(y)}{r}\Big)^{\alpha-1}\Big(\dashint_{D(0,r)}|\nabla u_\varepsilon|^2 dx\Big)^{\frac{1}{2}},
\end{aligned}
\end{equation*}
for any $y\in D(0,r)$, where $y^\prime\in\Delta(0,2r)$
is such that $\delta(y)=|y-y^\prime|$.
Integrating both sides with respect to $y$ in $D(0,r)$, we then have
\begin{equation*}
\int_{D(0,r)} \dashint_{B(y,\delta(y)/8)}|\nabla u_\varepsilon|^p dxdy
 \leq Cr^{d-1+(1-\alpha)p}
\int_0^r \frac{dt}{t^{(1-\alpha)p}}
\Big(\dashint_{D(0,2r)}|\nabla u_\varepsilon|^2\Big)^{\frac{p}{2}}
\leq Cr^{d-\frac{pd}{2}}\|\nabla u_\varepsilon\|_{L^2(D(0,2r))}^p,
\end{equation*}
whenever $\alpha\in(1-\frac{1}{p},1)$. We claim that
\begin{equation}\label{f:5.2}
 c_0\int_{D(0,r)}|\nabla u_\varepsilon|^p dx
 \leq \int_{D(0,r)} \dashint_{B(y,\delta(y)/8)}|\nabla u_\varepsilon|^p dxdy,
\end{equation}
and this implies the desired estimate $\eqref{pri:5.6}$. In order to prove $\eqref{f:5.2}$,
decomposing $D(0,r)$ as a non-overlapping union of cubes $\{Q\}$ by the Whitney decomposition theorem,
which satisfy the property that $3Q\subset D(0,2r)$, and the side length of $Q$ and
$\text{dist}(Q,\partial\Omega)$ are comparable. Observing that
\begin{equation*}
\begin{aligned}
\int_{Q} \dashint_{B(y,\delta(y)/8)}|\nabla u_\varepsilon|^p dxdy
&\geq c_0\dashint_Q \int_{B(y,l(Q)/8)}|\nabla u_\varepsilon|^p dx dy \\
&= c_0\dashint_Q |\nabla u_\varepsilon|^p\int_{B(y,l(Q)/8)} dy dx
\geq c_0\int_Q|\nabla u_\varepsilon|^p dx,
\end{aligned}
\end{equation*}
where $c_0$ depends only on $d$. This gives the estimate $\eqref{f:5.2}$, and we are done.
\end{proof}

\begin{thm}[A real method]\label{thm:5.2}
Let $q>2$ and $D_0\subset\Omega$ be a bounded Lipschitz domain. Let $F\in L^2(\Omega)$ and
$f\in L^p(\Omega)$ for some $2<p<q$. Suppose that for each ball $B$ with
the property that $|B|\leq c_0|D_0|$ and either $4B\subset\Omega$ or
$B$ is centered on $\partial\Omega$, there exist two measurable functions
$F_B$ and $R_B$ on $\Omega\cap 2B$, such that $|F|\leq |F_B|+ |R_B|$
on $\Omega\cap 2B$,
\begin{equation}\label{pri:5.8}
\begin{aligned}
\Big(\dashint_{\Omega\cap 2B}|R_B|^q\Big)^{1/q}
&\leq N_1 \bigg\{
\Big(\dashint_{\Omega\cap 4B}|F|^2\Big)^{1/2}
+\sup_{4D_0\supset B^\prime \supset B}
\Big(\dashint_{\Omega\cap B^\prime}|f|^2\Big)^{1/2}\bigg\},\\
\Big(\dashint_{\Omega\cap 2B}|F_B|^2\Big)^{1/2}
&\leq N_2 \sup_{4D_0\supset B^\prime \supset B}
\Big(\dashint_{\Omega\cap B^\prime}|f|^2\Big)^{1/2}
+\eta \Big(\dashint_{\Omega\cap 4B}|F|^2\Big)^{1/2},
\end{aligned}
\end{equation}
where $N_1,N_2>0$ and $0<c_0<1$. Then there exists
$\eta_0>0$, depending only on $N_1,N_2,c_0,p,q$ and the Lipschitz character of
$\Omega$, with the property that if $0\leq\eta < \eta_0$, then $F\in L^p_{loc}(\Omega)$
and
\begin{equation}\label{pri:5.9}
\Big(\dashint_{D_0}|F|^p\Big)^{1/p}
\leq C\bigg\{\Big(\dashint_{4D_0}|F|^2\Big)^{1/2}
+ \Big(\dashint_{4D_0}|f|^p\Big)^{1/p}\bigg\},
\end{equation}
where $C$ depends at most on
$N_1,N_2,c_0,p,q$ and the Lipschitz character of $\Omega$.
\end{thm}

\begin{proof}
See for example \cite[Theorem 3.2.6]{S4}.
This result was originally given by Z. Shen in \cite{S3}.
\end{proof}

\begin{thm}[boundary $W^{1,p}$ estimates]\label{thm:5.3}
Let $2\leq p<\infty$. Suppose that
$\mathcal{L}_\varepsilon$ satisfies
the assumptions $\eqref{a:1}$, $\eqref{a:2}$ and $\eqref{a:3}$.
Let $u_\varepsilon$ be
a weak solution of $L_\varepsilon u_\varepsilon = \emph{div}(f)+ F$
in $B(0,1)\cap\Omega$ with
$u_\varepsilon = g$ on $B(0,1)\cap\partial\Omega$,
where $f\in L^p(\Omega;\mathbb{R}^d)$, $F\in L^p(\Omega)$ and
$g\in W^{1-1/p,p}(\partial\Omega)$.
Then there holds
\begin{equation}\label{pri:5.7}
\begin{aligned}
\Big(\dashint_{D(0,r)}|\nabla u_\varepsilon|^p\Big)^{1/p}
\leq
& C_p
\Bigg\{\Big(\dashint_{D(0,4r)}|\nabla u_\varepsilon|^2\Big)^{1/2}
+\Big(\dashint_{D(0,4r)}|f|^p\Big)^{1/p} \\
&+r\Big(\dashint_{D(0,4r)}|F|^p\Big)^{1/p}
+ \Big(\dashint_{D(0,4r)}|\nabla \tilde{g}|^p\Big)^{1/p}\Bigg\}
\end{aligned}
\end{equation}
for any $0<r<(1/4)$,
where $C_p$ depends on $\mu_0, \mu_1, \mu_2, d, p$ and the character of $\Omega$.
\end{thm}

\begin{proof}
In terms of Theorem $\ref{thm:5.2}$, it suffices to verified the condition $\eqref{pri:5.8}$.
To do so, here
let $R_B = \nabla v_\varepsilon$, $F_B = \nabla u_\varepsilon - \nabla v_\varepsilon$,
and $\Omega(x,2r) = \Omega\cap B(x,2r)$ with $x\in D(0,r/2)$,  where
$v_\varepsilon\in H^1(\Omega(x,2r))$ satisfies
\begin{equation}\label{pde:5.1}
 \mathcal{L}_\varepsilon v_\varepsilon = 0 \quad\text{in~} \Omega(x,2r)
 \qquad v_\varepsilon = u_\varepsilon - \tilde{g}
 \quad\text{on~} \partial\Omega(x,2r).
\end{equation}
Thus, it follows from the estimate $\eqref{pri:5.6}$ that
\begin{equation}\label{f:5.5}
\begin{aligned}
\Big(\dashint_{\Omega(x,r)}|\nabla v_\varepsilon|^p\Big)^{1/p}
&\leq C_p\Big(\dashint_{\Omega(x,2r)}|\nabla v_\varepsilon|^2\Big)^{1/2} \\
&\leq C_p\bigg\{\Big(\dashint_{\Omega(x,2r)}|\nabla u_\varepsilon|^2\Big)^{1/2}
+ \Big(\dashint_{\Omega(x,2r)}|\nabla u_\varepsilon - \nabla v_\varepsilon|^2\Big)^{1/2}\bigg\}
\end{aligned}
\end{equation}
On the other hand, in view of $\eqref{pde:5.1}$ we may have
\begin{equation}\label{pde:5.2}
\int_{\Omega(x,2r)} \big[A(x/\varepsilon,\nabla u_\varepsilon)
-A(x/\varepsilon,\nabla v_\varepsilon)\big]\cdot\nabla\varphi
= -\int_{\Omega(x,2r)}f\cdot\nabla\varphi dx + \int_{\Omega(x,2r)} F\varphi dx
\end{equation}
for any $\varphi\in H_0^1(\Omega(x,2r))$.
By setting $\varphi = u_\varepsilon - v_\varepsilon - \tilde{g}$ in the above equation, we may
have
\begin{equation*}
\text{LHS~of~}\eqref{pde:5.2}
\geq \mu_0\int_{\Omega(x,2r)}|\nabla u_\varepsilon - \nabla v_\varepsilon|^2 dz
- \mu_2\int_{\Omega(x,2r)}|\nabla u_\varepsilon - \nabla v_\varepsilon||\nabla\tilde{g}|dz,
\end{equation*}
where we use the assumptions $\eqref{a:1},\eqref{a:2}$ and $\eqref{a:3}$,
and
\begin{equation*}
\text{RHS~of~}\eqref{pde:5.2}
\leq \Big(\frac{\mu_0}{4}+\frac{\mu_0}{8}\Big)\int_{\Omega(x,2r)}
|\nabla u_\varepsilon - \nabla v_\varepsilon|^2 dz
+ C\int_{\Omega(x,2r)}\Big( |f|^2 + r^2|F|^2 + |\nabla\tilde{g}|^2 \Big)dz,
\end{equation*}
where we also employ Poincar\'e's inequality and Young's inequality. Collecting the above estimates
leads to
\begin{equation}\label{f:5.4}
 \frac{\mu_0}{8}\int_{\Omega(x,2r)}|\nabla u_\varepsilon - \nabla v_\varepsilon|^2 dz
 \leq C(\mu_0,\mu_2,d)\int_{\Omega(x,2r)} |\tilde{f}|^2 dz
\end{equation}
where
$\tilde{f}(z) = |f(z)| + |F(z)|\text{dist}(z,M_{3r}^{4r}(x)) + |\nabla\tilde{g}(z)|$, which
is similar to that in $\eqref{def:2.4}$. Thus,
\begin{equation*}
\Big(\dashint_{\Omega(x,r)}|F_B|^2\Big)^{1/2}
\leq N_2(\mu_0,\mu_2,d)\Big(\dashint_{\Omega(x,2r)}|\tilde{f}|^2\Big)^{1/2}
\end{equation*}
and combining $\eqref{f:5.5}$ and $\eqref{f:5.4}$ gives
\begin{equation*}
\Big(\dashint_{\Omega(x,r)}|R_B|^p\Big)^{1/p}
\leq N_1(\mu_0,\mu_2,d)\bigg\{
\Big(\dashint_{\Omega(x,2r)}|\nabla u_\varepsilon|^2\Big)^{1/2}
+ \Big(\dashint_{\Omega(x,2r)}|\tilde{f}|^2\Big)^{1/2}\bigg\}.
\end{equation*}

Consequently, the stated estimate
$\eqref{pri:5.7}$ follows from Theorem $\ref{thm:5.2}$ and we have
completed the proof.
\end{proof}

\noindent\textbf{Proof of Theorem $\ref{thm:1.3}$}.
For any $x\in\overline{\Omega}$ and some $r>0$ we set
$\Omega(x,r)=B(x,r)\cap\Omega$.
In view of Theorems $\ref{thm:4.1}$ and $\ref{pri:5.3}$, one may have
\begin{equation*}
\begin{aligned}
\Big(\int_{\Omega(x,r)}|\nabla u_\varepsilon|^p dz\Big)^{1/p}
\leq
& C_p
\Bigg\{r^{\frac{d}{p}-\frac{d}{2}}\Big(\int_{\Omega(x,4r)}
|\nabla u_\varepsilon|^2dz\Big)^{1/2}
+\Big(\int_{\Omega(x,4r)}|f|^p\Big)^{1/p} \\
&+r\Big(\int_{\Omega(x,4r)}|F|^p\Big)^{1/p}
+ \Big(\int_{\Omega(x,4r)}|\nabla \tilde{g}|^p\Big)^{1/p}\Bigg\}.
\end{aligned}
\end{equation*}
By a covering argument, we further obtain
\begin{equation*}
\begin{aligned}
\|\nabla u_\varepsilon\|_{L^p(\Omega)}
&\leq C\bigg\{\|\nabla u_\varepsilon\|_{L^2(\Omega)}
+\|F\|_{L^p(\Omega)}+\|f\|_{L^p(\Omega)}+\|\nabla \tilde{g}\|_{L^p(\Omega)}\bigg\}\\
&\leq  C\bigg\{\|F\|_{L^p(\Omega)}+\|f\|_{L^p(\Omega)}
+\|g\|_{W^{1-1/p,p}(\partial\Omega)}\bigg\},
\end{aligned}
\end{equation*}
where we use the estimate $\eqref{pri:2.12}$ and the fact
$\|g\|_{H^{1/2}(\partial\Omega)}\leq C\|g\|_{W^{1-1/p,p}(\partial\Omega)}$
in the second inequality. Note that $\tilde{g}$ is an extension of $g$,
satisfying $\|\tilde{g}\|_{W^{1,p}(\Omega)}\leq
C\|g\|_{W^{1-1/p,p}(\partial\Omega)}$. We have completed the proof.
\qed

\begin{center}
\textbf{Acknowledgements}
\end{center}

The authors thank Prof. Zhongwei Shen for suggesting this topic
when he visited Lanzhou University last summer. The second author
wants to express his sincere appreciation to
Prof. Zhongwei Shen for his constant and illuminating instruction.
The first author and the last author was supported by the National Natural Science Foundation of China
(Grant NO. 11471147).
The second author was supported by the China Postdoctoral Science Foundation (Grant No. 2017M620490).

\end{document}